\UseRawInputEncoding

\documentclass[11pt]{amsart}
\usepackage{amsfonts}
\usepackage{amsmath}
\usepackage{amssymb,latexsym}
\usepackage[mathcal]{eucal}
\usepackage[pdftex,bookmarks,colorlinks,breaklinks]{hyperref}

\input xy
\xyoption{all}

\newcommand{\Z}{\mathbb{Z}}
\newcommand{\R}{\mathbb{R}}
\newcommand{\N}{\mathbb{N}}

\newcommand{\br}{\vspace{3 mm}}

\newcommand{\Aff}{{\rm{Aff\,}}}
\newcommand{\Inf}{{\rm{Inf\,}}}
\newcommand{\Homeo}{{\rm{Homeo\,}}}

\newcommand{\capp}{{\rm{cap}}}
\newcommand{\CO}{{\rm{CO}}}

\swapnumbers
\theoremstyle{plain}
\newtheorem{thm}{Theorem}[section]
\newtheorem{cor}[thm]{Corollary}
\newtheorem{lem}[thm]{Lemma}
\newtheorem{prop}[thm]{Proposition}

\theoremstyle{definition}
\newtheorem{defn}[thm]{Definition}
\newtheorem{rmk}[thm]{Remark}
\newtheorem{question}[thm]{Question}
\newtheorem{exa}[thm]{Example}

\begin{document}

\title[Affine Evaluation Maps, Dimension Groups, and Fair measures]
{Affine Evaluation Maps, Dimension Groups, and Fair measures}

\author{Eli Glasner}
\address{Department of Mathematics\\
Tel Aviv University\\
Tel Aviv\\
Israel}
\email{glasner@math.tau.ac.il}
\date{August 19, 2026}

\begin{abstract}
Let $X$ be a Cantor space and let $Q\subset M_{fc}(X)$ be a compact Choquet
simplex of atomless full-support probability measures.  We introduce the
associated \emph{Affine Evaluation Map} (AEM), which assigns to each clopen
set $A\subset X$ the affine function
\[
\widehat A(\mu)=\mu(A),\qquad \mu\in Q,
\]
and study the geometric subset condition obtained by comparing these
evaluation functions pointwise on $Q$.

For a good geometric AEM we construct the ordered group
\[
G_Q=C(X,\mathbb Z)/N_Q,
\qquad
N_Q=\left\{f\in C(X,\mathbb Z):
\int f\,d\mu=0\ \text{for every }\mu\in Q\right\},
\]
and show that it is a simple dimension group whose normalized state space is
canonically $Q$.  We prove that its order interval $[0,u]$ is precisely the
clopen scale and that
\[
J_Q=N_Q,
\]
where $J_Q$ is generated by the elementary relations
$\mathbf 1_A-\mathbf 1_B$ with $\widehat A=\widehat B$.  We also show that
the full stabilizer $\mathcal H_Q$ has invariant-measure simplex exactly
$Q$.

Using the clopen-scale property and the Herman--Putnam--Skau realization
theorem, we obtain a dimension-group proof that every good geometric AEM is
realized by a minimal Cantor homeomorphism $T$ with $Q=M_T(X)$.  For a
Cantor minimal system we identify $G_Q$ with the classical dimension group
modulo infinitesimals.  We further distinguish goodness, fairness,
ergodicity, and minimality of the measure stabilizer by explicit examples.

Finally, we apply the AEM framework to minimal Cantor actions of countable
amenable groups.  If $Q=M_G(X)$, then $Q$ canonically defines a proper
geometric AEM, and we clarify which parts of the preceding theory depend
only on $Q$ and which are specifically $\mathbb Z$-dynamical.  In
particular, $Q$ is good if and only if there exists a minimal homeomorphism
$T$ of $X$ such that
\[
M_T(X)=M_G(X).
\]
\end{abstract}



\maketitle

\tableofcontents

\setcounter{section}{0}
\setcounter{secnumdepth}{2}

\section{Introduction}

In the study of Cantor minimal systems, invariant probability measures and
ordered dimension groups play central roles in orbit-equivalence theory:
the ordered dimension group is the basic invariant for strong orbit
equivalence, while the simplex of invariant measures is fundamental for
orbit equivalence \cite{HPS-92,GW-95,GPS-95}.  When a system is uniquely
ergodic, the simplex reduces to a single measure $\mu$.  In a seminal paper
\cite{Akin-05}, E. Akin introduced the notion of a ``good measure'' on a
Cantor space and proved that any such measure can be realized as the unique
invariant measure of a minimal homeomorphism.  His subset condition was
motivated by the earlier work \cite{GW-95}, where the unique invariant
measure of a strictly ergodic Cantor system was shown to satisfy this
property.

In this paper we develop the corresponding framework for a spatially
embedded Choquet simplex $Q\subset M(X)$.  The associated \emph{Affine
Evaluation Map} (AEM) assigns to each clopen set $A\subset X$ the continuous
affine function
\[
\widehat A(\mu)=\mu(A),\qquad \mu\in Q.
\]
Our main tool is the \emph{subset condition}, an abstract Boolean property
which may be viewed as the measure-theoretic or affine shadow of dynamical
comparison (see Remark \ref{rmk-comp} below).

\begin{defn}\ 
Let $X$ denote the Cantor set.
\begin{enumerate}
  \item Let $M(X)$ denote the compact convex set of Borel probability
  measures on $X$, and let $M_{fc}(X)$ denote the set of atomless,
  full-support probability measures on $X$.
  \item
  For a homeomorphism $T$ of $X$ we let $M_T(X)$ denote the (compact, convex) Choquet simplex of $T$-invariant measures in $M(X)$.
  \item 
  Let $\mathcal{O}^+ = \CO(X)$ be the countable Boolean algebra of clopen subsets of the Cantor space $X$.
  Let $\mathcal{O} = \mathcal{O}^+ \setminus \{\emptyset, X\}$ be the collection of proper clopen subsets.
  \item 
  Given a Choquet simplex $Q \subset M(X)$
  \footnote{For many facts proven in the sequel about AEMs, one only needs to assume that $Q \subset M_{fc}(X)$ is
  nonempty, compact and convex. The assumption that $Q$ is a Choquet simplex is mainly used in the proof of 
  Theorem \ref{thm-GQ-dimension-group}, where we need the fact that $\Aff(Q)$ satisfies the Riesz interpolation property.}, we define the \emph{Affine Evaluation Map (AEM)} $\Phi = \Phi_Q$
  on $\CO(X)$ by assigning to each $A \in \CO(X)$ the continuous affine function $\widehat{A} \in \Aff(Q)$ given by:
  $$\Phi(A) = \hat{A}, \qquad \widehat{A}(\mu) = \mu(A) \quad \text{for all } \mu \in Q.$$
  
  \item 
  We define the pointwise strict order on the image of the AEM as follows: for $A, B \in \mathcal{O}^+$, we write $\widehat{A} < \widehat{B}$ if and only if $\mu(A) < \mu(B)$ for all $\mu \in Q$.
  \item We say the AEM generated by $Q$ is \emph{proper} if every $\mu \in Q$ has full support. This implies that for any $A \in \mathcal{O}$, we have $0 < \inf_{\mu \in Q} \mu(A) \le \sup_{\mu \in Q} \mu(A) < 1$.
  \item
  For a given AEM $\Phi$ we define the function 
  $$
  \capp(A) = \sup_{\mu \in Q}\mu(A) = \sup_{\mu \in Q} \hat{A}(\mu).
  $$
 By compactness of $Q$ we have $\capp(A) = \mu(A)$ for at least one measure $\mu \in Q$.
 Note that, with $B = A^c$,  $1 - \capp(A) = \inf_{\mu \in Q} \ \hat B(\mu) $.
\end{enumerate}
\end{defn}

As in \cite{GW-95} we denote by $FG(X, T)$ the full group of the Cantor minimal system $(X, T)$. The following result is from \cite{GW-95}.

\begin{thm}\label{thm-GW95}
Let $(X, T)$ be minimal and let $Q = M_T(X)$. If $\widehat{A} < \widehat{B}$, then there is a homeomorphism $h \in FG(X, T)$ such that $h(A) \subset B$. Denoting $A' = h(A)$, we then have that:
$$A' \subset B \quad \text{and} \quad \widehat{A'} = \widehat{A}.$$
\end{thm}

This motivates the following definition:

\begin{defn}
An AEM generated by a compact convex set $Q$ satisfies the \emph{subset condition} if for every $A, B \in \mathcal{O}^+$,
$$\widehat{A} < \widehat{B} \implies \exists A' \subset B \quad \text{with} \quad \widehat{A'} = \widehat{A}.$$
In this case we say that the AEM on $Q$ is good, or simply that {\em $Q$ is good}.
\end{defn}

\begin{lem}\label{lem-quasi}
Let $Q \subset M_{fc}(X)$ define a proper, good AEM. Then, 
\begin{enumerate}
\item
if $\widehat{A} < \widehat{B}$, there is $D \in \mathcal{O}$ with:
$$\widehat{D} = \widehat{B} - \widehat{A}.$$
\item
If $A, B, C \in \mathcal{O}$ satisfy $\widehat{A} = \widehat{B} + \widehat{C}$, then there is a clopen partition $A = A_B \sqcup A_C$ such that $\widehat{A_B} = \widehat{B}$ and $\widehat{A_C} = \widehat{C}$.
This extends immediately by induction to any finite sum $\widehat{A} = \sum_{m} \widehat{B_m}$.
\end{enumerate}
\end{lem}

\begin{proof}
(1)\ 
In the definition of the subset condition, since $\widehat{A} < \widehat{B}$, there exists $A' \subset B$ such that $\widehat{A'} = \widehat{A}$. Take $D = B \setminus A'$. Since $A' \subset B$, we have $\mathbf{1}_B = \mathbf{1}_{A'} + \mathbf{1}_D$.

Therefore, evaluating any $\mu \in Q$ yields:
$$\mu(D) = \mu(B) - \mu(A') = \mu(B) - \mu(A).$$
Thus, $\widehat{D} = \widehat{B} - \widehat{A}$.

(2)\ 
 Since $\widehat{B} < \widehat{B} + \widehat{C} = \widehat{A}$, the subset condition implies there exists $A_B \subset A$ with $\widehat{A_B} = \widehat{B}$. Letting $A_C = A \setminus A_B$, Lemma \ref{lem-quasi} guarantees $\widehat{A_C} = \widehat{A} - \widehat{A_B} = \widehat{C}$.  
\end{proof}

\br

\begin{question}\label{qu-good-dynamical}
Does the converse of Theorem \ref{thm-GW95} hold? I.e., if $Q \subset M(X)$ is a Choquet simplex that defines a good AEM, is it of the form $Q = M_T(X)$, for some Cantor minimal $(X, T)$?
\end{question}

\begin{rmk}
By Akin's theorem this is always the case when $Q = \{\mu\}$ with $\mu$ a good measure. Moreover, it then follows that the corresponding $T$ can be chosen so that the system $(X, \mu, T)$ is minimal and uniquely ergodic.
\end{rmk}

\begin{rmk}
By a theorem of Downarowicz \cite{Dow-91} every abstract Choquet simplex can be realized as $M_T(X)$ for a Toeplitz Cantor minimal system $(X, T)$. The above question is of a different nature. We are given a specific spatial embedding of a Choquet simplex $Q \subset M(X)$ and ask whether it is identically $M_T(X)$ for some minimal $T \in \Homeo(X)$.
\end{rmk}

\begin{defn}
We say that an AEM is \emph{geometric} when it is defined by a Choquet simplex $Q \subset M(X)$, and that it is \emph{dynamical} when $Q = M_T(X)$ for a minimal dynamical system $(X, T)$.
\end{defn}

Our aim is to give an AEM and dimension-group proof of the affirmative
answer to Question \ref{qu-good-dynamical}.  We first show that a good AEM
yields a simple dimension group, and then apply the
Herman--Putnam--Skau construction \cite{HPS-92} to obtain a
Bratteli--Vershik system realizing $Q$ as its invariant-measure simplex
(Theorem \ref{thm-realization}).  This method generalizes the
dimension-group proof of Akin's theorem suggested in \cite{Gl-02}.

The realization theorem itself is not new; related 
realization results were proved in \cite{Dah-08,IM-17,Melleray-19}.  The
point here is the AEM formulation and the accompanying dimension-group
approach.  In particular, the clopen-scale theorem and the identity
$J_Q=N_Q$ make explicit how the Boolean relations between clopen sets are
encoded by affine evaluation.  We believe this gives a useful perspective
on the interaction between invariant measures, dimension groups, and
comparison.

\begin{rmk}\label{rmk-comp}
The notion of ``comparison'' was initiated in \cite{GW-95} and \cite{Gl-02}, although not in this name. In recent years, the notion of ``strict comparison'' has become a central topic in topological dynamics, playing a vital role in the classification of $C^*$-crossed products, see \cite{Matui-12, K-20, KS-20}.
\end{rmk}

In Sections \ref{sec-fair} and \ref{sec-stabilizer}
we define a measure $\mu$ to be {\em fair} if there exists some $T \in \Homeo(X)$ such that the system
$(X,T)$ is minimal and $\mu$ is $T$-invariant. We then 
distinguish goodness, fairness,
ergodicity, and minimality of the measure stabilizer by explicit examples.

\br

An advantage of the AEM formulation is that most of the structural theory
depends only on the embedded simplex $Q\subset M(X)$ and not on a
distinguished acting transformation.  This allows us, in the final
sections, to apply the theory to minimal Cantor actions
$G\curvearrowright X$ of countable amenable groups.  For such an action the
invariant-measure simplex $Q=M_G(X)$ is a compact Choquet simplex of
atomless full-support measures.  If $Q$ satisfies the subset condition, all
of the preceding AEM conclusions apply unchanged.  The HPS and GPS results
enter only afterwards, through an auxiliary minimal $\mathbb Z$-action
realizing the same invariant simplex; they do not describe the orbit
structure of the original $G$-action.  In particular, we show that
$M_G(X)$ is good if and only if there exists a minimal homeomorphism $T$ of
$X$ such that $M_T(X)=M_G(X)$.  This reformulates, in AEM language, a
natural open problem concerning invariant measures of minimal amenable
Cantor actions.

\begin{rmk}
We note that there is a close connection between the present point of view
and the work of Bezuglyi and Handelman \cite{BH-14}, who studied goodness
of measures, and more generally goodness of sets of traces, from the
perspective of dimension groups.  In particular, several phenomena
concerning good and non-good invariant measures which arise below also
appear in their work.  Our emphasis here is different: we start with a
spatially embedded simplex $Q\subset M(X)$ and its affine evaluation map
on the Boolean algebra of clopen subsets of $X$, and develop the associated
dimension-group structure directly from the subset condition.  
\end{rmk}

\br

{\bf Acknowledgement:}
I thank Ethan Akin for reading a draft of this work and suggesting some improvements.
While preparing this work I was helped, both mathematically and editorially, by ChatGPT plus.
Of course, I am responsible for the validity of all the proofs in this paper.

\br

\section{A Uniform Clopen Rokhlin Lemma}\label{sec-uniform-RL}

The following elementary uniform clopen Rokhlin lemma provides the dynamical model for the Boolean construction of the next section (see  Glasner and Weiss \cite{GW-06}).

\begin{thm}[Uniform Clopen Rokhlin Lemma]\label{thm-UR}
Let $(X, T)$ be a minimal dynamical system where $X$ is a Cantor space. Given a positive integer $n$ and a real number $\epsilon > 0$, there exists a clopen subset $B \subset X$ such that the sets $B, TB, \dots, T^{n-1}B$ are pairwise disjoint, and for every $T$-invariant probability measure $\mu \in M_T(X)$, the remainder $R = X \setminus \bigcup_{j=0}^{n-1} T^j B$ satisfies:
$$\mu(R) < \epsilon.$$
\end{thm}

\begin{proof}
Fix an integer $N > n/\epsilon$. Because $X$ is an infinite Cantor space and $T$ is minimal, the system is aperiodic. Therefore, for any $x \in X$, the points $x, Tx, \dots, T^{N-1}x$ are distinct. By the continuity of $T$ and the fact that $X$ is zero-dimensional, we can find a clopen neighborhood $U$ of $x$ such that the iterates $U, TU, \dots, T^{N-1}U$ are pairwise disjoint.

Consider the first return time function to $U$, denoted $r_U : U \to \mathbb{N}$, where:
$$r_U(y) = \min\{m \ge 1 : T^m y \in U\}.$$
Because $X$ is minimal, every orbit returns to $U$, so $r_U$ is finite everywhere. Because $U$ is clopen and $T$ is continuous, $r_U$ is locally constant. Since $U$ is compact, $r_U$ takes only finitely many values. This yields a finite clopen partition of the base $U = \bigsqcup_{k \in K} U_k$, where $U_k = \{y \in U : r_U(y) = k\}$ and $K$ is a finite set of positive integers.

This defines a finite Kakutani-Rokhlin partition of $X$ into clopen sets:
$$X = \bigsqcup_{k \in K} \bigsqcup_{j=0}^{k-1} T^j U_k.$$
Because the first $N$ iterates of $U$ are disjoint, the minimum return time to $U$ is at least $N$. Thus, $k \ge N$ for all $k \in K$.

We now construct the Rokhlin base $B$ by packing blocks of height $n$ into each column of the Kakutani skyscraper. For each column $U_k$, let $m_k = \lfloor k/n \rfloor$ be the number of full $n$-blocks that fit into the column. Define the base for this column as:
$$B_k = \bigcup_{m=0}^{m_k-1} T^{mn} U_k,$$
and set the global base $B = \bigsqcup_{k \in K} B_k$.

Because $B$ is a finite union of clopen sets, it is clopen. By construction, the sets $B, TB, \dots, T^{n-1}B$ simply translate the $n$-blocks strictly upward within their respective columns, never spilling over the top because $m_k n \le k$. Therefore, $B, TB, \dots, T^{n-1}B$ are pairwise disjoint.

Finally, we estimate the measure of the remainder $R = X \setminus \bigcup_{j=0}^{n-1} T^j B$. In each column $k$, the untiled remainder at the top of the column has height $r_k = k - n\lfloor k/n \rfloor$, which is strictly less than $n$. Thus, the remainder is contained within the top $n - 1$ levels of each column:
$$R \subset \bigcup_{k \in K} \bigcup_{j=k-r_k}^{k-1} T^j U_k.$$

Let $\mu \in M_T(X)$ be any invariant measure. Since $\mu$ is $T$-invariant, the measure of each level in a column is $\mu(U_k)$. The total measure of the remainder is bounded by:
$$\mu(R) \le \sum_{k \in K} (n-1)\mu(U_k) < n \sum_{k \in K} \mu(U_k) = n\mu(U).$$

To bound $\mu(U)$, we use the fact that the columns partition $X$ and each column has height $k \ge N$:
$$1 = \mu(X) = \sum_{k \in K} k\mu(U_k) \ge \sum_{k \in K} N\mu(U_k) = N\mu(U).$$
Thus, $\mu(U) \le 1/N$.

It follows that for every invariant measure $\mu \in M_T(X)$,
$$\mu(R) < n\mu(U) \le \frac{n}{N} < \epsilon.$$
This completes the proof.
\end{proof}

The point of the preceding result for us is motivational. In the abstract setting of a good AEM there is no distinguished transformation $T$. Nevertheless, the subset condition allows us to replace the levels $T^j B$
 by mutually $Q$-equivalent clopen sets. The next section develops this Boolean analogue of the Rokhlin construction.

\section{A Boolean Rokhlin lemma}\label{sec-BRL}

Our aim in this section is as follows: given a Choquet simplex $Q \subset M_{fc}(X)$ such that the associated AEM is good, we would like to define a simple dimension group. In order to apply the Effros-Handelman-Shen theorem \cite{EHS-80}, we must construct an ordered, simple, unperforated abelian group with the Riesz interpolation property from the AEM.
 
\begin{lem}[Uniform Atomlessness]\label{lem:uniform_atomless}
Let $X$ be a Cantor space and $Q \subset M_{fc}(X)$ a compact Choquet simplex of non-atomic measures. For every $\epsilon > 0$, there exists a finite partition of $X$ into mutually disjoint clopen sets $P_1, \dots, P_M$ such that 
$\capp(P_j)=\sup_{\mu \in Q} \mu(P_j) < \epsilon$ for all $1 \le j \le M$.
\end{lem}

\begin{proof}
Fix $\epsilon > 0$. Because every measure in $Q$ is non-atomic, for every $x \in X$ and every $\mu \in Q$, there exists a clopen neighborhood $U_{x,\mu}$ of $x$ such that $\mu(U_{x,\mu}) < \epsilon$.

Since the evaluation map $\nu \mapsto \nu(U_{x,\mu})$ is weakly continuous and $Q$ is compact, there exists a weak$^*$ open neighborhood $W_\mu$ of $\mu$ in $Q$ such that $\nu(U_{x,\mu}) < \epsilon$ for all $\nu \in W_\mu$. By the compactness of $Q$, finitely many such open sets $W_{\mu_1}, \dots, W_{\mu_k}$ cover $Q$. 

Let $U_x = \bigcap_{i=1}^k U_{x,\mu_i}$. Then $U_x$ is a clopen neighborhood of $x$ such that $\sup_{\nu \in Q} \nu(U_x) < \epsilon$. The collection $\{U_x\}_{x \in X}$ forms an open cover of the compact space $X$. We extract a finite subcover $U_{x_1}, \dots, U_{x_m}$. 

Let $P_1, \dots, P_M$ be the disjointified clopen partition generated by the intersections of this finite subcover. By construction, the partition consists of mutually disjoint clopen subsets of $X$, and $\sup_{\mu \in Q} \mu(P_j) < \epsilon$ for all $1 \le j \le M$.
\end{proof}

\begin{defn}
For clopen sets $A, B \subset X$, we define the relation $A \sim_Q B$ to mean that
$\widehat{A}(\mu)=\widehat{B}(\mu)$ for every $\mu\in Q$.
Thus $\sim_Q$ is an equivalence relation on $\CO(X)$.
\end{defn}

The following homogeneity observation is the link between the Boolean formulation
and the Rokhlin-tower argument, Proposition \ref{prop-uniform-R} below.
See the proof of \cite[Proposition 2.4]{Gl-02}.

\medskip
\begin{prop}[Boolean homogeneity]\label{prop-homog} 
If $A,B\in \CO(X)$ and $A\sim_Q B$, then there exists a homeomorphism
$h\in\Homeo(X)$ such that
\[
h(A)=B
\quad\text{and}\quad
h_*\mu=\mu \qquad (\mu\in Q).
\]
\end{prop}

\begin{proof}
We first construct a homeomorphism from $A$ onto $B$ which preserves the
restrictions of all the measures in $Q$.

Enumerate the clopen subsets of $A$ and of $B$.  By a back-and-forth
construction, build finite clopen partitions $\mathcal P_n$ of $A$ and
$\mathcal Q_n$ of $B$, together with bijections between their atoms, so that
corresponding atoms are $Q$-equivalent, each successive pair of partitions
refines the preceding pair, and the first $n$ clopen sets in each enumeration
are unions of atoms of the corresponding partition.

Indeed, suppose that $P\in\mathcal P_n$ corresponds to $Q_0\in\mathcal Q_n$
and that a new clopen set splits $P$ as $P=P_1\sqcup P_2$.  Since
\[
\widehat{Q_0}=\widehat P=\widehat{P_1}+\widehat{P_2},
\]
Lemma~\ref{lem-quasi}(2) gives a clopen partition
$Q_0=Q_1\sqcup Q_2$ with $Q_i\sim_Q P_i$.  The same argument, with the roles
of $A$ and $B$ reversed, gives the back step.  Empty pieces are simply
ignored.

The limiting correspondence is an isomorphism between the Boolean algebras
$\CO(A)$ and $\CO(B)$ which preserves every evaluation $\mu\in Q$.  By Stone
duality it is induced by a homeomorphism $h_A:A\to B$, and
\[
\mu(h_A^{-1}(C))=\mu(C)
\]
for every clopen $C\subset B$ and every $\mu\in Q$.
Since $A^c\sim_Q B^c$, the same construction gives a homeomorphism
$h_{A^c}:A^c\to B^c$ preserving all the restricted measures.  Gluing these
two homeomorphisms gives the required $h\in\Homeo(X)$.
\end{proof}

Let
\[
\mathcal{H}_Q=\{h\in\Homeo(X):h_*\mu=\mu\text{ for every }\mu\in Q\}.
\]

\begin{lem}\label{lem-minimal}
Every $\mathcal{H}_Q$-orbit is dense in $X$; i.e. $(X, \mathcal{H}_Q)$ is minimal.
\end{lem}

\begin{proof}
Fix $x\in X$ and a nonempty clopen set $V$.  Since all measures in $Q$ have
full support and $Q$ is compact,
\[
\delta=\min_{\mu\in Q}\mu(V)>0.
\]
Apply Lemma~\ref{lem:uniform_atomless} with $\epsilon=\delta$, and let $P$ be
the atom of the resulting partition which contains $x$.  Then
\[
\widehat P<\widehat V.
\]
By the subset condition there is a clopen set $C\subset V$ such that
$C\sim_Q P$.  the Boolean homogeneity claim gives an element
$h\in \mathcal{H}_Q$ with $h(P)=C$.  In particular, $h(x)\in V$.
\end{proof}

\begin{defn}
Fix an integer $L\geq1$.  An \emph{$L$-dividing partition} of a clopen
set $C$ is a finite clopen partition
\[
C=\bigsqcup_i\bigsqcup_{j=0}^{m_i-1} C_{i,j},
\qquad m_i\geq L,
\]
arranged into columns
\[
\mathcal C_i=(C_{i,0},\ldots,C_{i,m_i-1}),
\]
such that, for every $i$ and $j$, there exists
$g_{i,j}\in\mathcal H_Q$ satisfying
\[
g_{i,j}(C_{i,0})=C_{i,j}.
\]\end{defn}

\begin{lem}[Dividing-partition claim]\label{lem-dividing}
Every clopen set $C$ admits an $L$-dividing partition.
\end{lem}

\begin{proof}
We use two elementary observations.

First, an $L$-dividing partition of a clopen set $C$ may be refined
compatibly with any prescribed finite clopen partition
\[
C=P_1\sqcup\cdots\sqcup P_t.
\]
Indeed, suppose
\[
\mathcal C_i=(C_{i,0},\ldots,C_{i,m_i-1})
\]
is one of the columns, with
\[
C_{i,j}=g_{i,j}(C_{i,0}).
\]
On the base $C_{i,0}$ consider the finite family of clopen sets
\[
g_{i,j}^{-1}(P_\ell\cap C_{i,j}),
\qquad
0\leq j<m_i,\quad 1\leq\ell\leq t,
\]
and let $\mathcal R_i$ be the finite clopen partition of $C_{i,0}$
generated by this family.  For every atom $E\in\mathcal R_i$, replace
the original column by
\[
(E,g_{i,1}(E),\ldots,g_{i,m_i-1}(E)).
\]
The resulting columns still have height $m_i\geq L$, and every one of
their levels is contained in some $P_\ell$.  Carrying this out for all
columns gives the desired refinement.

Second, if clopen sets $C$ and $D$ admit $L$-dividing partitions, then so does
$C\cup D$.  By induction on the number of columns in a partition of $D$, it
is enough to treat the case in which $D$ is one column
\[
D=V_0\sqcup\cdots\sqcup V_{m-1},
\qquad V_j=k_j(V_0),\quad k_j\in \mathcal{H}_Q,
\]
with $m\geq L$ and $k_0$ the identity.

Partition $V_0$ by the finite family of clopen sets
$k_j^{-1}(C)$, $0\leq j<m$.  Thus, for every atom $E$ of this partition and
every $j$, the set $k_j(E)$ is either contained in $C$ or disjoint from $C$.
Refine the given $L$-dividing partition of $C$ so that every $k_j(E)$ which
is contained in $C$ is a union of atoms of the refined partition.

If none of the sets $k_j(E)$ is contained in $C$, then
\[
(k_0(E),\ldots,k_{m-1}(E))
\]
is a new column, disjoint from $C$.

Refine the given $L$-dividing partition of $C$ so that every set
$k_j(E)$ which is contained in $C$ is a union of levels of the refined
dividing partition:

Suppose that $k_{j_0}(E)\subset C$.  Write
\[
k_{j_0}(E)=F_1\sqcup\cdots\sqcup F_t,
\]
where each $F_\nu$ is a level of the refined $L$-dividing partition of
$C$.  For each such level $F_\nu$ and each $j$ for which
$k_j(E)\cap C=\emptyset$, append
\[
k_jk_{j_0}^{-1}(F_\nu)
\]
as a new level to the column containing $F_\nu$.

These added levels are
pairwise disjoint, are disjoint from $C$, and are $\mathcal{H}_Q$-equivalent to $F$.
Carrying this out for every atom $E$ produces an $L$-dividing partition of
$C\cup D$.

We now prove the lemma.  Let $x\in C$.  By
the minimality claim, Lemma \ref{lem-minimal}, the orbit $\mathcal{H}_Qx$ is dense and therefore contains
at least $L$ distinct points of $C$.  Choose
$h_0,\ldots,h_{L-1}\in \mathcal{H}_Q$, with $h_0$ the identity, so that the points
$h_j(x)$ are distinct and belong to $C$.  There is a clopen neighborhood
$U_x$ of $x$ such that the sets
\[
h_0(U_x),\ldots,h_{L-1}(U_x)
\]
are pairwise disjoint and contained in $C$.  Their union is a clopen
neighborhood of $x$ admitting a one-column $L$-dividing partition.
Compactness gives finitely many such neighborhoods covering $C$, and repeated
application of the union observation gives an $L$-dividing partition of
$C$.
\end{proof}

\begin{prop}[Boolean Rokhlin lemma]\label{prop-uniform-R}
Let $Q$ generate a proper, good geometric AEM, and let $A$ be a clopen set.
For any integer $N\geq1$ and any $\epsilon>0$, there exist mutually disjoint
clopen sets $B_1,\ldots,B_N\subset A$ such that
$B_i\sim_Q B_1$ for all $i$, and such that the remainder
$R=A\setminus\bigcup_{i=1}^N B_i$
satisfies
\[
\capp(R)=\sup_{\mu\in Q}\mu(R)<\epsilon.
\]
\end{prop}

\begin{proof}
The assertion is immediate when $A=\emptyset$.  Choose an integer
$L>N/\epsilon$, and let
\[
\mathcal C_i=(C_{i,0},\ldots,C_{i,m_i-1})
\]
be an $L$-dividing partition of $A$, supplied by
the dividing-partition claim.  For each $i$, write
\[
m_i=Nq_i+r_i,
\qquad 0\leq r_i<N.
\]
For $1\leq s\leq N$, define
\[
B_s=\bigsqcup_i\ \bigsqcup_{t=0}^{q_i-1}C_{i,Nt+s-1},
\]
and let $R$ be the union of the remaining $r_i$ levels in each column.
The sets $B_1,\ldots,B_N,R$ form a clopen partition of $A$.

All levels in the $i$-th column are $Q$-equivalent.  Hence, for every
$1\leq s\leq N$,
\[
\widehat{B_s}=\sum_i q_i\widehat{C_{i,0}},
\]
so $B_s\sim_Q B_1$.

Finally, for every $\mu\in Q$,
\[
\mu(R)=\sum_i r_i\mu(C_{i,0})
\leq (N-1)\sum_i\mu(C_{i,0}).
\]
Since $m_i\geq L$,
\[
\mu(A)=\sum_i m_i\mu(C_{i,0})
\geq L\sum_i\mu(C_{i,0}),
\]
and therefore
\[
\mu(R)\leq\frac{N-1}{L}\mu(A)
\leq\frac{N-1}{L}<\epsilon.
\]
Taking the supremum over $\mu\in Q$ completes the proof.
\end{proof}

\br

\section{The ordered group \texorpdfstring{$G_Q$}{GQ}}\label{sec-DG}

For $f\in C(X,\Z)$ define its {\em evaluation profile} by
\[
\widehat f(\mu)=\int_X f\,d\mu,
\qquad \mu\in Q,
\]
and let
\[
N_Q
=
\left\{
f\in C(X,\Z):
\widehat f(\mu)=0
\text{ for every }\mu\in Q
\right\}.
\]
We define
\[
G_Q=C(X,\Z)/N_Q,
\qquad
u=[\mathbf 1_X].
\]

If $g=[f]\in G_Q$, then
\[
\widehat g:=\widehat f
\]
is well defined. Equivalently, the evaluation homomorphism identifies
$G_Q$ with the countable subgroup
\[
V_Q
=
\left\{
\widehat f:f\in C(X,\Z)
\right\}
\subset \Aff(Q).
\]

In particular, for clopen sets $A,B\subset X$,
\[
[\mathbf 1_A]=[\mathbf 1_B]
\quad\Longleftrightarrow\quad
\widehat A=\widehat B.
\]
Thus equality in $G_Q$ records exactly the geometric equivalence of
clopen sets induced by the AEM.

We equip $G_Q$ with the strict pointwise positive cone
\[
G_Q^+
=
\{0\}
\cup
\left\{
g\in G_Q:
\widehat g(\mu)>0
\text{ for every }\mu\in Q
\right\}.
\]
Accordingly,
\[
g<h
\quad\Longleftrightarrow\quad
\widehat g(\mu)<\widehat h(\mu)
\text{ for every }\mu\in Q.
\]

\begin{prop}\label{prop-basic-GQ}
The triple $(G_Q,G_Q^+,u)$ is a countable, simple, unperforated
ordered abelian group with order unit.
\end{prop}

\begin{proof}
Since $X$ is a Cantor space, its Boolean algebra of clopen sets is
countable. Hence $C(X,\Z)$, and therefore also $G_Q$, is countable.

The evaluation map identifies $G_Q$ with a subgroup of $\Aff(Q)$.
Consequently, $G_Q$ is torsion free. Moreover, $G_Q^+$ is closed
under addition and
\[
G_Q^+\cap(-G_Q^+)=\{0\},
\]
so it is a positive cone.

Let $g\in G_Q$. Since $\widehat g$ is continuous on the compact set
$Q$, it is bounded. Choose $n\in\N$ such that
\[
-n<\widehat g(\mu)<n
\qquad
\text{for every }\mu\in Q.
\]
Then
\[
-nu<g<nu,
\]
and hence $u$ is an order unit.

To prove unperforation, suppose that
$ng\in G_Q^+$
for some $n\geq1$. If $ng=0$, then $g=0$, since $G_Q$ is torsion
free. Otherwise,
$
n\widehat g(\mu)>0 \ \text{for every }\mu\in Q$,
and therefore $\widehat g(\mu)>0 \ 
\text{for every }\mu\in Q$.
Thus, $g\in G_Q^+$.

Finally, let $0\neq g\in G_Q^+$. By compactness,
$\delta = \min_{\mu\in Q}\widehat g(\mu)>0$.
Choose $m\in\N$ such that $m\delta>1$. Then
$u<mg$.
Since $u$ is an order unit, it follows that $g$ is itself an order
unit. Thus every nonzero positive element of $G_Q$ is an order unit,
and $G_Q$ is simple.
\end{proof}

The Boolean Rokhlin lemma now supplies the approximate divisibility
which is needed to prove that $G_Q$ is sufficiently large inside
$\Aff(Q)$.

\begin{lem}\label{lem-Udense}
Suppose that $Q$ generates a proper, good geometric AEM. Then
\[
V_Q
=
\left\{
\widehat f:f\in C(X,\Z)
\right\}
\]
is uniformly dense in $\Aff(Q)$.
\end{lem}

\begin{proof}
We first show that the functions
\[
a_\varphi(\mu)
=
\int_X\varphi\,d\mu,
\qquad
\varphi\in C(X,\R),
\]
have uniformly dense span in $\Aff(Q)$.

Suppose otherwise. By the Hahn--Banach theorem, there would exist a
nonzero continuous linear functional
\[
\Lambda:\Aff(Q)\longrightarrow\R
\]
which vanishes on every $a_\varphi$. Extending $\Lambda$ to $C(Q)$
and applying the Riesz representation theorem, we obtain a finite
signed measure $\lambda$ on $Q$ such that
\[
\Lambda(a)=\int_Q a(\mu)\,d\lambda(\mu)
\qquad
\text{for every }a\in\Aff(Q).
\]

For every $\varphi\in C(X,\R)$ we then have
\[
0
=
\int_Q\left(\int_X\varphi\,d\mu\right)d\lambda(\mu).
\]
Taking $\varphi=\mathbf 1_X$ shows that $\lambda(Q)=0$.

Write the Jordan decomposition as
\[
\lambda=\lambda^+-\lambda^-.
\]
The measures $\lambda^+$ and $\lambda^-$ have the same total mass,
say $c$. If $c=0$, then $\lambda=0$, contrary to the choice of
$\Lambda$. Thus $c>0$.

Let $\mu_+,\mu_-\in Q$ be the barycenters of
$c^{-1}\lambda^+$ and $c^{-1}\lambda^-$, respectively. The preceding
identity gives
\[
c\int_X\varphi\,d\mu_+
=
c\int_X\varphi\,d\mu_-
\]
for every $\varphi\in C(X,\R)$. Hence $\mu_+=\mu_-$.

For every $a\in\Aff(Q)$, the barycenter identity now gives
\[
\Lambda(a)
=
c\,a(\mu_+)-c\,a(\mu_-)
=
0,
\]
contradicting the fact that $\Lambda$ is nonzero. Therefore the
functions $a_\varphi$ are uniformly dense in $\Aff(Q)$.

Since $X$ is zero-dimensional, the locally constant real-valued
functions are uniformly dense in $C(X,\R)$. It follows that the real
linear span of
\[
\left\{
\widehat A:A\subset X\text{ clopen}
\right\}
\]
is uniformly dense in $\Aff(Q)$.

It remains to prove that every real multiple of every clopen
evaluation function belongs to the closure of $V_Q$.

Fix a clopen set $A\subset X$, an integer $N\geq1$, and
$\epsilon>0$. By the Boolean Rokhlin lemma, Proposition
\ref{prop-uniform-R}, there exist mutually disjoint clopen sets
\[
B_1,\ldots,B_N\subset A
\]
such that
$B_i\sim_Q B_1 \ (1\leq i\leq N)$,
and such that
$R = A\setminus\bigsqcup_{i=1}^N B_i$
satisfies
$\sup_{\mu\in Q}\mu(R)<N\epsilon$.

Since the $B_i$ are $Q$-equivalent,
\[
\widehat A
=
N\widehat{B_1}+\widehat R.
\]
Consequently,
\[
\left\|
\widehat{B_1}
-
\frac1N\widehat A
\right\|_\infty
=
\frac1N\|\widehat R\|_\infty
<
\epsilon.
\]
Since $\widehat{B_1}\in V_Q$, this proves that
\[
\frac1N\widehat A\in\overline{V_Q}.
\]

It follows that every rational multiple of $\widehat A$ belongs to
$\overline{V_Q}$. By approximation of real numbers by rationals, every
real multiple of $\widehat A$ also belongs to $\overline{V_Q}$.
Therefore $\overline{V_Q}$ contains the real linear span of all
clopen evaluation functions. Hence
\[
\overline{V_Q}=\Aff(Q).
\]
\end{proof}

\begin{thm}\label{thm-GQ-dimension-group}
Suppose that $Q$ generates a proper, good geometric AEM. Then
$(G_Q,G_Q^+,u)$ is a simple dimension group. Moreover, its normalized
state space is canonically affinely homeomorphic to $Q$.
\end{thm}

\begin{proof}
By Proposition \ref{prop-basic-GQ}, the group is countable, simple,
and unperforated, and $u$ is an order unit. It remains to establish
the Riesz interpolation property.

Let
\[
x_1,x_2,y_1,y_2\in G_Q
\]
satisfy
\[
x_i\leq y_j
\qquad
(i,j\in\{1,2\}).
\]

If $x_i=y_j$ for some pair $(i,j)$, then that common element is an
interpolant. Indeed, the other lower bound is below $y_j$, and the
common element $x_i$ is below the other upper bound.

We may therefore assume that all four inequalities are strict. Set
\[
\rho
=
\min_{\substack{\mu\in Q\\ i,j\in\{1,2\}}}
\left(
\widehat y_j(\mu)-\widehat x_i(\mu)
\right).
\]
By compactness of $Q$, we have $\rho>0$.

Since $Q$ is a Choquet simplex, $\Aff(Q)$ has the Riesz interpolation
property. Hence there exists $a\in\Aff(Q)$ such that
\[
\widehat x_i+\frac{\rho}{3}
\leq a
\leq
\widehat y_j-\frac{\rho}{3}
\qquad
(i,j\in\{1,2\}).
\]

By Lemma \ref{lem-Udense}, choose $z\in G_Q$ such that
\[
\|\widehat z-a\|_\infty<\frac{\rho}{3}.
\]
It follows that
\[
\widehat x_i(\mu)
<
\widehat z(\mu)
<
\widehat y_j(\mu)
\]
for every $\mu\in Q$ and every $i,j\in\{1,2\}$. By the definition of
the order on $G_Q$,
\[
x_i\leq z\leq y_j
\qquad
(i,j\in\{1,2\}).
\]
Thus $G_Q$ has the Riesz interpolation property and is therefore a
simple dimension group.

It remains to identify the normalized state space. For each
$\mu\in Q$, define
\[
s_\mu:G_Q\longrightarrow\R,
\qquad
s_\mu(g)=\widehat g(\mu).
\]
This is a normalized positive group homomorphism, since
$s_\mu(u)=1$.

Thus, we obtain an affine continuous map
\[
Q\longrightarrow S(G_Q,u),
\qquad
\mu\longmapsto s_\mu.
\]
It is injective because two Borel probability measures on the Cantor
space which agree on every clopen set are equal.

Conversely, let $s\in S(G_Q,u)$. We first show that $s$ is continuous
with respect to the uniform norm inherited from $\Aff(Q)$. 

If $\|\widehat g\|_\infty<1/n$,
then $-u<ng<u$.
Positivity of $s$ gives
$-1\leq ns(g)\leq 1$,
and hence
$|s(g)|\leq\frac1n$.

Thus, $s$ is uniformly continuous and, by Lemma \ref{lem-Udense},
extends uniquely to a continuous linear functional
\[
\widetilde s:\Aff(Q)\longrightarrow\R.
\]

The extension is positive. Indeed, suppose that
$a\in\Aff(Q)$ satisfies $a\geq0$. For every $\epsilon>0$, choose
$g_\epsilon\in G_Q$ such that
\[
\left\|
\widehat{g_\epsilon}
-
(a+\epsilon\mathbf 1)
\right\|_\infty
<
\frac{\epsilon}{2}.
\]
Then,
\[
\widehat{g_\epsilon}(\mu)>\frac{\epsilon}{2}
\qquad
\text{for every }\mu\in Q,
\]
so $g_\epsilon\in G_Q^+$ and therefore
$s(g_\epsilon)\geq0$. Letting $\epsilon\downarrow0$ gives
$\widetilde s(a)\geq 0 $.
Also,
$\widetilde s(\mathbf 1) = 1$.

A normalized positive linear functional on $\Aff(Q)$ is evaluation
at a unique point of $Q$. Indeed, extend $\widetilde s$ positively to
$C(Q)$, represent the extension by a probability measure on $Q$, and
take its barycenter. Thus there exists a unique $\mu\in Q$ such that
\[
\widetilde s(a)=a(\mu)
\qquad
\text{for every }a\in\Aff(Q).
\]
In particular,
\[
s(g)=\widehat g(\mu)=s_\mu(g)
\qquad
\text{for every }g\in G_Q.
\]
Therefore every normalized state is of the form $s_\mu$, and
$Q\cong S(G_Q,u)$ canonically as compact convex sets.
\end{proof}

\section{The clopen scale}\label{subsec-clopen-scale}

For the realization argument we need one additional fact: the order interval
$[0,u]$ in $G_Q$ consists exactly of classes of clopen sets.  We prove this
concretely, using only the elementary relations between clopen sets with the
same $Q$-evaluation profile.  The same argument will also show that these
relations generate the full kernel $N_Q$ of the evaluation map.

Let
\[
J_Q=
\left\langle
\mathbf 1_A-\mathbf 1_B:
A,B\in \CO(X),\ A\sim_QB
\right\rangle
\subset C(X,\Z),
\]
and set
\[
G_Q^J:=C(X,\Z)/J_Q,
\qquad
u_J:=[\mathbf 1_X]_{J_Q}.
\]
We equip $G_Q^J$ with the cone
\begin{equation}\label{eq-GJ-cone}
(G_Q^J)^+
=
\{[f]_{J_Q}:f\in C(X,\Z),\ f\geq0\}.
\end{equation}
Since $J_Q\subseteq N_Q$, there is a canonical surjective homomorphism
\begin{equation}\label{eq-qJ}
q_J:G_Q^J\longrightarrow G_Q,
\qquad
q_J([f]_{J_Q})=[f]_{N_Q}.
\end{equation}

\begin{lem}\label{lem-GJ-order}
The set in \eqref{eq-GJ-cone} is a positive cone on $G_Q^J$, and $u_J$ is an
order unit.  Moreover, the map $q_J$ in \eqref{eq-qJ} is positive.
\end{lem}

\begin{proof}
The cone is clearly closed under addition.  Suppose that
\[
[f]_{J_Q}\in (G_Q^J)^+\cap -(G_Q^J)^+.
\]
Then there are $p,q\in C(X,\Z)$ with $p,q\geq0$ such that
\[
[f]_{J_Q}=[p]_{J_Q}=-[q]_{J_Q}.
\]
Hence $p+q\in J_Q\subseteq N_Q$.  Therefore
\[
\int_X(p+q)\,d\mu=0
\qquad(\mu\in Q).
\]
Since every measure in $Q$ has full support and $p+q\geq0$, this forces
$p+q=0$, and hence $p=q=0$.  Thus the cone is proper.

If $f\in C(X,\Z)$ and $n\geq\|f\|_\infty$, then
$n\mathbf 1_X\pm f\geq0$, so
$-n u_J\leq[f]_{J_Q}\leq n u_J$.
Thus, $u_J$ is an order unit.

Finally, if $[f]_{J_Q}\geq0$, choose $p\geq0$ with
$[f]_{J_Q}=[p]_{J_Q}$.  If $p=0$, then $q_J([f]_{J_Q})=0$.  If $p\neq0$,
full support gives
\[
\int_Xp\,d\mu>0
\qquad(\mu\in Q),
\]
so $q_J([f]_{J_Q})$ is strictly positive in $G_Q$.  Hence $q_J$ is positive.
\end{proof}

We next record a measure-theoretic consequence of Boolean homogeneity.  Put
\[
M_{\mathcal H_Q}(X)
=
\{\nu\in M(X):h_*\nu=\nu\text{ for every }h\in\mathcal H_Q\}.
\]

\begin{lem}[Uniform $\mathcal H_Q$-smallness]\label{lem-HQ-small}
For every $\epsilon>0$ there exists $\delta>0$ such that, for every clopen
$A\subset X$,
$\sup_{\mu\in Q}\mu(A)<\delta$ implies $\nu(A)<\epsilon$
for every $\nu\in M_{\mathcal H_Q}(X)$.
\end{lem}

\begin{proof}
Choose $N$ with $1/N<\epsilon$.  Apply Proposition~\ref{prop-uniform-R} to
$X$ with this $N$ and with remainder of $Q$-capacity less than $1/2$.  We
obtain pairwise disjoint, mutually $Q$-equivalent nonempty clopen sets
$B_1,\ldots,B_N$.  Put
\[
\delta=\min_{\mu\in Q}\mu(B_1)>0.
\]
If $\sup_{\mu\in Q}\mu(A)<\delta$, then
$\widehat A<\widehat{B_1}$.  The subset condition gives a clopen
$A'\subset B_1$ with $A'\sim_QA$.  By Proposition~\ref{prop-homog}, there
is $h\in\mathcal H_Q$ with $h(A)=A'$.  Hence every
$\nu\in M_{\mathcal H_Q}(X)$ satisfies
\[
\nu(A)=\nu(A')\leq\nu(B_1).
\]
Likewise, Proposition~\ref{prop-homog} shows that
\[
\nu(B_1)=\cdots=\nu(B_N).
\]
Since the $B_i$ are pairwise disjoint,
$N\nu(B_1)\leq1$,
and therefore $\nu(A)\leq1/N<\epsilon$.
\end{proof}

\begin{lem}[Clopen approximation]\label{lem-HQ-clopenization}
Let $f:X\to[0,1]$ be locally constant and let $\epsilon>0$.  There exists a
clopen set $C\subset X$ such that
\[
\left|\nu(C)-\int_Xf\,d\nu\right|<\epsilon
\]
for every $\nu\in M_{\mathcal H_Q}(X)$.
\end{lem}

\begin{proof}
Write
\[
X=P_1\sqcup\cdots\sqcup P_m,
\qquad
f|_{P_i}=t_i\in[0,1].
\]
Choose $N$ so large that $1/N<\epsilon/2$, and apply
Lemma~\ref{lem-HQ-small} with $\epsilon/(2m)$ to obtain $\delta>0$.
For each $i$, Proposition~\ref{prop-uniform-R} gives a decomposition
\[
P_i=P_{i,1}\sqcup\cdots\sqcup P_{i,N}\sqcup R_i,
\qquad
P_{i,j}\sim_QP_{i,1},
\]
with
$\sup_{\mu\in Q}\mu(R_i)<\delta$.

Hence Lemma~\ref{lem-HQ-small} gives
\[
\nu(R_i)<\frac{\epsilon}{2m}
\qquad
(\nu\in M_{\mathcal H_Q}(X)).
\]
Choose $k_i\in\{0,\ldots,N\}$ with
$|k_i/N-t_i|\leq1/N$, and put
\[
C=\bigsqcup_i\bigsqcup_{j=1}^{k_i}P_{i,j}.
\]
For every $\nu\in M_{\mathcal H_Q}(X)$, Boolean homogeneity gives equal
$\nu$-measure to the $N$ levels over $P_i$, and hence
\[
\nu(P_{i,j})=
\frac{\nu(P_i)-\nu(R_i)}{N}.
\]
Consequently,
\[
\begin{aligned}
\left|\nu(C)-\int_Xf\,d\nu\right|
&\leq
\sum_i\left|\frac{k_i}{N}-t_i\right|\nu(P_i)
+\sum_i\frac{k_i}{N}\nu(R_i)\\
&\leq\frac1N+\sum_i\nu(R_i)
<\epsilon.
\end{aligned}
\]
\end{proof}

\begin{prop}\label{prop-HQ-measures}
The set of $\mathcal H_Q$-invariant probability measures is exactly $Q$:
\[
M_{\mathcal H_Q}(X)=Q.
\]
\end{prop}

\begin{proof}
The inclusion $Q\subseteq M_{\mathcal H_Q}(X)$ follows from the definition of
$\mathcal H_Q$.  Conversely, let $\nu\in M_{\mathcal H_Q}(X)$ and suppose
that $\nu\notin Q$.  Since $Q$ is compact and convex, Hahn--Banach separation,
followed by a sufficiently close locally constant approximation and an affine
rescaling, gives a locally constant function $f:X\to[0,1]$ such that
\[
\beta:=\int_Xf\,d\nu
>
\alpha:=\sup_{\mu\in Q}\int_Xf\,d\mu.
\]
Put $r=(\alpha+\beta)/2$ and choose
\[
0<\epsilon<\frac{\beta-\alpha}{4}.
\]
By Lemma~\ref{lem-HQ-clopenization}, choose clopen sets $C,D$ which
approximate, uniformly over $M_{\mathcal H_Q}(X)$, respectively the functions
$f$ and the constant function $r$.  Thus, for every $\mu\in Q$,
\[
\mu(C)<\alpha+\epsilon<r-\epsilon<\mu(D),
\]
whereas
\[
\nu(C)>\beta-\epsilon>r+\epsilon>\nu(D).
\]
Hence $\widehat C<\widehat D$ on $Q$.  The subset condition gives a clopen
$C'\subset D$ with $C'\sim_QC$.  Proposition~\ref{prop-homog} gives
$h\in\mathcal H_Q$ with $h(C)=C'$.  Since $\nu$ is $\mathcal H_Q$-invariant,
\[
\nu(C)=\nu(C')\leq\nu(D),
\]
a contradiction.  Therefore $\nu\in Q$.
\end{proof}

\begin{cor}\label{cor-HQ-jordan}
If a finite signed measure belongs to $\operatorname{span}_{\R}Q$, then its
positive and negative parts in the Jordan decomposition are nonnegative scalar
multiples of measures in $Q$.
\end{cor}

\begin{proof}
Every element of $\operatorname{span}_{\R}Q$ is $\mathcal H_Q$-invariant.
By uniqueness of the Jordan decomposition, its positive and negative parts
are also $\mathcal H_Q$-invariant.  After normalization,
Proposition~\ref{prop-HQ-measures} applies.
\end{proof}

The next lemma is the key approximation statement.

\begin{lem}[Relative clopen density]\label{lem-relative-density}
Let $A\subset X$ be clopen.  Then
\[
\left\{\widehat C:C\subset A\text{ clopen}\right\}
\]
is uniformly dense in the affine order interval
\[
\left\{c\in\Aff(Q):0\leq c\leq\widehat A\right\}.
\]
\end{lem}

\begin{proof}
For $\varphi\in C(X,\R)$ write
\[
a_\varphi(\mu)=\int_X\varphi\,d\mu.
\]
We first show that the functions $a_\varphi$ with
$0\leq\varphi\leq\mathbf 1_A$ are dense in the indicated affine interval.
Suppose that $0\leq c\leq\widehat A$ is not in their uniform closure.
Hahn--Banach separation gives a continuous linear functional $\Lambda$ on
$\Aff(Q)$ such that
\begin{equation}\label{eq-Lambda}
\Lambda(c)>
\sup_{0\leq\varphi\leq\mathbf 1_A}\Lambda(a_\varphi).
\end{equation}
The functional
\[
\varphi\longmapsto\Lambda(a_\varphi)
\]
on $C(X,\R)$ is bounded, and therefore is represented by a finite signed
measure $\sigma$ on $X$:
\[
\int_X\varphi\,d\sigma=\Lambda(a_\varphi).
\]

Extend $\Lambda$ to a continuous linear functional on $C(Q)$.  By the
Riesz representation theorem, there is a finite signed measure
$\lambda$ on $Q$ such that
\[
\Lambda(d)=\int_Q d(\nu)\,d\lambda(\nu)
\qquad(d\in\Aff(Q)).
\]
Write
$\lambda=\lambda^+-\lambda^-$
for its Jordan decomposition, and put
\[
\alpha=\lambda^+(Q),\qquad \beta=\lambda^-(Q).
\]
If $\alpha>0$, let $\mu_+\in Q$ be the barycenter of the probability
measure $\alpha^{-1}\lambda^+$; thus
\[
d(\mu_+)
=
\frac1\alpha\int_Q d(\nu)\,d\lambda^+(\nu)
\qquad(d\in\Aff(Q)).
\]
Similarly, if $\beta>0$, let $\mu_-\in Q$ be the barycenter of
$\beta^{-1}\lambda^-$.  These barycenters belong to $Q$ because $Q$ is
compact and convex.

For every $\varphi\in C(X,\R)$ we therefore have
\[
\begin{aligned}
\int_X\varphi\,d\sigma
&=\Lambda(a_\varphi)
=\int_Q a_\varphi(\nu)\,d\lambda(\nu)\\
&=\alpha a_\varphi(\mu_+)-\beta a_\varphi(\mu_-)
=\int_X\varphi\,d(\alpha\mu_+-\beta\mu_-).
\end{aligned}
\]
Hence
\[
\sigma=\alpha\mu_+-\beta\mu_-,
\]
with the evident omission of a term when $\alpha=0$ or $\beta=0$.
Consequently,
$\sigma\in\operatorname{span}_{\R}Q$.

By Corollary~\ref{cor-HQ-jordan}, the Jordan decomposition of $\sigma$ has the
form
\[
\sigma=\alpha\mu-\beta\nu,
\qquad
\alpha,\beta\geq0,
\quad
\mu,\nu\in Q,
\]
with the evident interpretation if one part is zero.  Since the functions
$a_\varphi$, $\varphi\in C(X,\R)$, are dense in $\Aff(Q)$ by the first part of
the proof of Lemma~\ref{lem-Udense},
\[
\Lambda(r)=\alpha r(\mu)-\beta r(\nu)
\qquad(r\in\Aff(Q)).
\]
Therefore
\[
\Lambda(c)
\leq\alpha c(\mu)
\leq\alpha\mu(A)
=\sigma^+(A).
\]
On the other hand, by the defining property of the positive variation,
\[
\sup_{0\leq\varphi\leq\mathbf 1_A}\Lambda(a_\varphi)
=
\sup_{0\leq\varphi\leq\mathbf 1_A}\int_X\varphi\,d\sigma
=
\sigma^+(A),
\]
contradicting (\ref{eq-Lambda}).

It remains to replace $\varphi$ by a clopen subset of $A$.  Approximate
$\varphi$ uniformly by a locally constant function
\[
\psi=\sum_{i=1}^m t_i\mathbf 1_{P_i},
\qquad
A=P_1\sqcup\cdots\sqcup P_m,
\quad
0\leq t_i\leq1,
\]
with $\psi=0$ off $A$.  For large $N$, apply
Proposition~\ref{prop-uniform-R} inside each $P_i$:
\[
P_i=P_{i,1}\sqcup\cdots\sqcup P_{i,N}\sqcup R_i,
\qquad
P_{i,j}\sim_QP_{i,1},
\]
with $\sup_{\mu\in Q}\mu(R_i)$ arbitrarily small.  Choose $k_i$ with
$|k_i/N-t_i|\leq1/N$ and put
\[
C=\bigsqcup_i\bigsqcup_{j=1}^{k_i}P_{i,j}\subset A.
\]
Then, uniformly in $\mu\in Q$,
\[
\left|\mu(C)-\int_X\psi\,d\mu\right|
\leq
\frac1N+\sum_i\sup_{\nu\in Q}\nu(R_i),
\]
which can be made arbitrarily small.  This proves the lemma.
\end{proof}

We now extend the subset condition from comparison of individual clopen sets to comparison of finite sums of clopen classes. All order comparisons in the next two results are taken in the concrete
ordered group $G_Q^J$.

\begin{lem}[One-to-many comparison]\label{lem-one-many}
Let $A,B_1,\ldots,B_s$ be nonempty  clopen sets.  If
\[
\widehat A<\sum_{j=1}^s\widehat{B_j}
\quad\text{pointwise on }Q,
\]
then
\[
[\mathbf 1_A]_{J_Q}
\leq
\sum_{j=1}^s[\mathbf 1_{B_j}]_{J_Q}
\quad\text{in }G_Q^J.
\]
\end{lem}

\begin{proof}
We use induction on $s$.  The case $s=1$ follows directly from the subset
condition: a $Q$-equivalent copy of $A$ can be placed inside $B_1$, and the
remaining part of $B_1$ represents the positive difference.

Assume the result for $s-1$ and put
\[
a=\widehat A,
\qquad
b=\widehat{B_1},
\qquad
c=\sum_{j=2}^s\widehat{B_j}.
\]

Put
$\rho=\min_{\mu\in Q}(b(\mu)+c(\mu)-a(\mu))>0$.
By compactness and full support,
\[
\min_Q a>0,
\qquad
\min_Q b>0,
\qquad
\min_Q c>0.
\]
Choose $\eta>0$ so small that
$4\eta<\min\{\min_Qa,\min_Qb,\min_Qc,\rho\}$.
Then, the two affine functions
\[
\eta\mathbf 1,
\qquad
a-c+2\eta\mathbf 1
\]
are both strictly below each of
\[
a-\eta\mathbf 1,
\qquad
b-\eta\mathbf 1.
\]
Since $Q$ is a Choquet simplex, $\Aff(Q)$ has Riesz interpolation.  Hence
there exists $d\in\Aff(Q)$ satisfying
\[
\eta\leq d\leq a-\eta,
\qquad
a-c+2\eta\leq d\leq b-\eta.
\]
By Lemma~\ref{lem-relative-density}, choose a clopen $C\subset A$ with
\[
\|\widehat C-d\|_\infty<\eta/2.
\]
Then
\[
\widehat C<\widehat{B_1},
\qquad
\widehat{A\setminus C}<\sum_{j=2}^s\widehat{B_j}.
\]
The subset condition gives
$
[\mathbf 1_C]_{J_Q}\leq[\mathbf 1_{B_1}]_{J_Q}$,
and the induction hypothesis gives
\[
[\mathbf 1_{A\setminus C}]_{J_Q}
\leq
\sum_{j=2}^s[\mathbf 1_{B_j}]_{J_Q}.
\]
Adding the two inequalities proves the result.
\end{proof}

\begin{prop}[Finite-sum strict comparison]\label{prop-stable-comparison}
Let $A_1,\ldots,A_r,B_1,\ldots,B_s$ be clopen sets.  If
\[
\sum_{i=1}^r\widehat{A_i}
<
\sum_{j=1}^s\widehat{B_j}
\quad\text{pointwise on }Q,
\]
then
\[
\sum_{i=1}^r[\mathbf 1_{A_i}]_{J_Q}
\leq
\sum_{j=1}^s[\mathbf 1_{B_j}]_{J_Q}
\quad\text{in }G_Q^J.
\]
\end{prop}

\begin{proof}
We induct on $r$.  The case $r=1$ is Lemma~\ref{lem-one-many}.  Let
\[
x=[\mathbf 1_{A_1}]_{J_Q},
\qquad
y=\sum_{j=1}^s[\mathbf 1_{B_j}]_{J_Q}.
\]
Since $\widehat{A_1}<\sum_j\widehat{B_j}$,
Lemma~\ref{lem-one-many} gives $x\leq y$.  By the definition of the positive
cone on $G_Q^J$, there is a nonnegative $h\in C(X,\Z)$ such that
\[
y-x=[h]_{J_Q}.
\]
Write
$h=\sum_{k=1}^m\mathbf 1_{C_k}$
for clopen sets $C_k$ (for example, take the nonempty level sets
$C_k=\{h\geq k\}$).  Since $J_Q\subseteq N_Q$, the preceding equality gives
\[
\sum_{k=1}^m\widehat{C_k}
=
\sum_{j=1}^s\widehat{B_j}-\widehat{A_1}.
\]
The original strict inequality therefore implies
\[
\sum_{i=2}^r\widehat{A_i}
<
\sum_{k=1}^m\widehat{C_k}.
\]
By the induction hypothesis,
\[
\sum_{i=2}^r[\mathbf 1_{A_i}]_{J_Q}
\leq
[h]_{J_Q}=y-x.
\]
Adding $x$ completes the proof.
\end{proof}

We shall also use the following elementary finite form of the subset
condition.

\begin{lem}[Finite packing]\label{lem-finite-packing}
Let $A_1,\ldots,A_r,B$ be clopen sets and suppose
\[
\sum_{i=1}^r\widehat{A_i}<\widehat B
\quad\text{pointwise on }Q.
\]
Then there exist pairwise disjoint clopen sets
\[
A_1',\ldots,A_r'\subset B
\]
such that $A_i'\sim_QA_i$ for every $i$.
\end{lem}

\begin{proof}
We place the sets successively.  Suppose that pairwise disjoint
$A_1',\ldots,A_{k-1}'\subset B$ have already been chosen with
$A_i'\sim_QA_i$, and put
\[
R_{k-1}=B\setminus\bigsqcup_{i<k}A_i'.
\]
Then
\[
\widehat{R_{k-1}}
=
\widehat B-\sum_{i<k}\widehat{A_i}.
\]
The assumed strict inequality gives
$\widehat{A_k}<\widehat{R_{k-1}}$.
The subset condition therefore provides a clopen
$A_k'\subset R_{k-1}$ with $A_k'\sim_QA_k$.  Induction completes the
construction.
\end{proof}

We can now identify the clopen scale in $G_Q$.

\begin{prop}[Clopen scale]\label{prop-clopen-scale}
For every good geometric AEM,
\[
\{[\mathbf 1_A]:A\in \CO(X)\}
=
[0,u]_{G_Q}.
\]
More generally, if $B$ is clopen and
\[
0\leq g\leq[\mathbf 1_B]
\]
in $G_Q$, then there is a clopen $A\subset B$ such that
\[
g=[\mathbf 1_A].
\]
\end{prop}

\begin{proof}
Only the inclusion $[0,u]\subseteq\{[\mathbf 1_A]:A\in \CO(X)\}$ needs
proof.  Let $0<g<u$ and choose $f\in C(X,\Z)$ representing $g$.  Write
$f=f^+-f^-$.
Since $\widehat f>0$ on $Q$,
$\widehat{f^-}<\widehat{f^+}$.

If $f^-=0$, then $[f]_{J_Q}\geq0$ by definition.  Otherwise, express the
nonnegative integer-valued functions $f^-$ and $f^+$ as finite sums of
characteristic functions of clopen sets.  Proposition~\ref{prop-stable-comparison}
then gives
\[
[f^-]_{J_Q}\leq[f^+]_{J_Q}.
\]
Thus in all cases $[f]_{J_Q}\geq0$ in $G_Q^J$.  By the definition of the
positive cone, there exists $h\in C(X,\Z)$ with $h\geq0$ such that
\[
[f]_{J_Q}=[h]_{J_Q}.
\]
Because $J_Q\subseteq N_Q$,
$\widehat h=\widehat f$.
Moreover, $g<u$ in $G_Q$, and therefore
$\widehat h=\widehat f<1
\quad\text{pointwise on }Q$.

Write
$h=\sum_{i=1}^r\mathbf 1_{A_i}$
with clopen $A_i$.  Then
\[
\sum_{i=1}^r\widehat{A_i}<\widehat X.
\]
By Lemma~\ref{lem-finite-packing}, there are pairwise disjoint clopen sets
$A_i'\subset X$ with $A_i'\sim_QA_i$.  Put
$A=\bigsqcup_{i=1}^rA_i'$.
Then
\[
[h]_{J_Q}
=
\sum_i[\mathbf 1_{A_i}]_{J_Q}
=
\sum_i[\mathbf 1_{A_i'}]_{J_Q}
=
[\mathbf 1_A]_{J_Q}.
\]
Applying $q_J$ gives
$g=[\mathbf 1_A] \ \text{in }G_Q$.
The endpoints $0$ and $u$ are represented by $\emptyset$ and $X$.

For the relative assertion, the cases $g=0$ and
$g=[\mathbf 1_B]$ are trivial.  Otherwise the global assertion gives a
clopen $C$ with $g=[\mathbf 1_C]$, and
$\widehat C<\widehat B$.
The subset condition gives a clopen $A\subset B$ with $A\sim_QC$, hence
$[\mathbf 1_A]=g$.
\end{proof}

\begin{cor}[Identification of the relation subgroup]\label{cor-JQ-NQ}
For every good geometric AEM,
\[
J_Q=N_Q.
\]
Consequently,
\[
G_Q^J=C(X,\Z)/J_Q
\quad\text{and}\quad
G_Q=C(X,\Z)/N_Q
\]
are canonically the same ordered group with order unit.
\end{cor}

\begin{proof}
We already know that $J_Q\subseteq N_Q$.  To prove the reverse inclusion, let
$f\in N_Q$.  The case $f=0$ is trivial.  Write
\[
f=f^+-f^-.
\]
By full support, both $f^+$ and $f^-$ are nonzero.  Since they are
nonnegative integer-valued locally constant functions, there are nonempty
clopen sets $A_1,\ldots,A_r$ and $B_1,\ldots,B_s$ such that
\[
f^+=\sum_{i=1}^r\mathbf 1_{A_i},
\qquad
f^-=\sum_{j=1}^s\mathbf 1_{B_j}.
\]
Because $f\in N_Q$, we have
\[
\sum_{i=1}^r[\mathbf 1_{A_i}]
=
\sum_{j=1}^s[\mathbf 1_{B_j}]
\qquad\text{in }G_Q.
\]
By Theorem~\ref{thm-GQ-dimension-group}, $G_Q$ is a dimension group and
therefore has the Riesz decomposition property.

Repeated application of
Riesz decomposition gives positive elements
\[
g_{ij}\in G_Q^+,
\qquad
1\leq i\leq r,
\quad
1\leq j\leq s,
\]
such that
\[
[\mathbf 1_{A_i}]=\sum_{j=1}^s g_{ij}
\quad(1\leq i\leq r),
\qquad
[\mathbf 1_{B_j}]=\sum_{i=1}^r g_{ij}
\quad(1\leq j\leq s).
\]

We now use the relative assertion of Proposition~\ref{prop-clopen-scale}.
For each fixed $i$, apply it successively inside the successive remainders of
$A_i$.  This gives a clopen partition
\[
A_i=\bigsqcup_{j=1}^s C_{ij}
\]
such that
\[
[\mathbf 1_{C_{ij}}]=g_{ij}
\qquad(1\leq j\leq s).
\]
Likewise, for each fixed $j$, we obtain a clopen partition
\[
B_j=\bigsqcup_{i=1}^r D_{ij}
\]
with
$[\mathbf 1_{D_{ij}}]=g_{ij}
\ (1\leq i\leq r)$.
Hence
$[\mathbf 1_{C_{ij}}]=[\mathbf 1_{D_{ij}}]
\ \text{in }G_Q$.
By the definition of $G_Q$, this is equivalent to
\[
\widehat{C_{ij}}=\widehat{D_{ij}},
\]
so $C_{ij}\sim_QD_{ij}$.  Therefore
\[
\mathbf 1_{C_{ij}}-\mathbf 1_{D_{ij}}\in J_Q
\]
for every $i,j$.  Since the above decompositions are clopen partitions,
\[
\begin{aligned}
f
&=\sum_{i=1}^r\mathbf 1_{A_i}
  -\sum_{j=1}^s\mathbf 1_{B_j}\\
&=\sum_{i=1}^r\sum_{j=1}^s
  \bigl(\mathbf 1_{C_{ij}}-\mathbf 1_{D_{ij}}\bigr)
\in J_Q.
\end{aligned}
\]
Thus $N_Q\subseteq J_Q$, and hence $J_Q=N_Q$.

It remains only to compare the orders.  Lemma~\ref{lem-GJ-order} shows that
the canonical group isomorphism $q_J:G_Q^J\to G_Q$ is positive.  Conversely,
let $0<g\in G_Q^+$ and choose $f\in C(X,\Z)$ representing $g$.  Since
$\widehat f>0$, the same stable-comparison argument used in the proof of
Proposition~\ref{prop-clopen-scale} gives
$
[f]_{J_Q}\geq0$.
Thus, $q_J^{-1}$ is positive as well.  The order units correspond by
construction, so $q_J$ is an order-unit isomorphism.
\end{proof}

In particular, one may equivalently write
\[
G_Q\cong C(X,\Z)/J_Q;
\]
that is, the full kernel of affine evaluation is generated by the elementary
relations coming from pairs of clopen sets with the same $Q$-evaluation profile.

\section{Infinitesimals and Orbit Equivalence}

In classical Bratteli-Vershik theory, the dimension group $K_0(X, T)$ classifies Strong Orbit Equivalence. It possesses an infinitesimal subgroup, defined as the kernel of all traces:
$$\Inf(K_0(X, T)) = \{g \in K_0(X, T) : \mu(g) = 0 \text{ for all } \mu \in M_T(X)\}.$$

The following proposition identifies the AEM dimension group $G_Q$ with the
classical dimension group modulo its infinitesimal subgroup, rendering it the algebraic invariant for Orbit Equivalence. 

\begin{prop}\label{prop:k-theory-quotient}
Let $(X, T)$ be a Cantor minimal system and let $Q = M_T(X)$ be the corresponding dynamical simplex. Let $K_0(X, T)$ be the classical dimension group of the system, and $G_Q$ the AEM dimension group. Then there is a canonical order-isomorphism:
$$G_Q \cong K_0(X, T)/\Inf(K_0(X, T)).$$
\end{prop}

\begin{proof}
By standard results in Cantor dynamics, the $K_0$-group of the system is canonically isomorphic to the quotient of the continuous integer-valued functions by the subgroup of coboundaries:
\[
K_0(X,T) \cong \frac{C(X, \mathbb{Z})}{B_T}, \quad \text{where} \quad B_T = \{f - f \circ T^{-1} : f \in C(X, \mathbb{Z})\}.
\]
Because every measure $\mu \in Q$ is $T$-invariant, the integral of any coboundary with respect to $\mu$ is zero. Therefore, $B_T \subseteq N_Q$, where $N_Q = \{f \in C(X, \mathbb{Z}) : \int f \, d\mu = 0 \text{ for all } \mu \in Q\}$.

The subgroup of infinitesimals in $K_0(X,T)$ with respect to the state space $Q$ consists exactly of those classes whose evaluations vanish on $Q$. Under the isomorphism above, this subgroup corresponds to the quotient $N_Q / B_T$.

Applying the Third Isomorphism Theorem for groups, we can quotient out the infinitesimals:
\[
\frac{K_0(X,T)}{\Inf(K_0(X,T))} \cong \frac{C(X, \mathbb{Z}) / B_T}{N_Q / B_T} \cong \frac{C(X, \mathbb{Z})}{N_Q}.
\]
By definition of our AEM dimension group, we have:
\[
G_Q = \frac{C(X, \mathbb{Z})}{N_Q}.
\]
Since the canonical positive cones and order units match under these identifications, this algebraic isomorphism is in fact an order-isomorphism, completing the proof.
\end{proof}

\section{Realization of the System}\label{sec-realization}

As was shown in \cite[Proposition 2.4]{Gl-02} two good measures $\mu$ and $\nu$ are equivalent iff $S(\mu) = S(\nu)$, where $S(\mu) = \{\mu(A) : A \in \mathcal{O}^+\}$. We will need the following analogue for good Choquet simplices.

\begin{lem}[Clopen-Lifting Lemma]\label{lem:clopen_lifting}
Let $X$ and $Y$ be Cantor spaces, and let $Q \subset M_{fc}(X)$ and $K \subset M_{fc}(Y)$ be good Choquet simplices. Suppose there is an ordered group isomorphism $\Phi : G_Q \to G_K$ preserving the order units, meaning $\Phi([X]) = [Y]$. Then there exists a homeomorphism $h : X \to Y$ such that $[h(A)] = \Phi([A])$ for every clopen set $A \subset X$
and such that $h_*(Q) = K$.
\end{lem}

\begin{proof}
By Proposition~\ref{prop-clopen-scale}, an order-unit isomorphism carries
clopen classes to clopen classes, and the relative form of that proposition
allows such a class to be represented inside any prescribed clopen upper
bound.

We now repeat the Boolean back-and-forth construction of
Proposition~\ref{prop-homog}.  Enumerate the clopen subsets of $X$ and $Y$.
Inductively construct finite clopen partitions $\mathcal P_n$ of $X$ and
$\mathcal Q_n$ of $Y$, together with bijections between their atoms, such that
corresponding atoms $P\in\mathcal P_n$ and $Q_0\in\mathcal Q_n$ satisfy
\[
[\mathbf 1_{Q_0}]=\Phi([\mathbf 1_P]).
\]
Suppose in a forth step that an atom $P$ is split as
\[
P=P_0\sqcup P_1,
\]
and let $Q_0$ be the corresponding atom.  Then
\[
0\leq\Phi([\mathbf 1_{P_0}])\leq[\mathbf 1_{Q_0}].
\]
By the relative clopen-scale property in $G_K$, there is a clopen
$Q_0'\subset Q_0$ such that
\[
[\mathbf 1_{Q_0'}]=\Phi([\mathbf 1_{P_0}]).
\]
Putting $Q_1'=Q_0\setminus Q_0'$ gives
\[
[\mathbf 1_{Q_1'}]
=[\mathbf 1_{Q_0}]-[\mathbf 1_{Q_0'}]
=\Phi([\mathbf 1_{P_1}]).
\]
Thus every splitting on the $X$ side can be matched by a clopen splitting on
the $Y$ side.  The back step is identical, using $\Phi^{-1}$.

At the limit we obtain a Boolean algebra isomorphism
\[
\Psi : \CO(X)\longrightarrow \CO(Y)
\]
with
\[
[\mathbf 1_{\Psi(A)}]=\Phi([\mathbf 1_A])
\qquad(A\in \CO(X)).
\]
By Stone duality, $\Psi$ is induced by a homeomorphism $h:X\to Y$, and hence
\[
[h(A)]=\Phi([A])
\]
for every clopen $A\subset X$.
\end{proof}

\begin{thm}\label{thm-good-dynamical}\label{thm-realization}
Let $Q \subset M_{fc}(X)$ be a Choquet simplex of full-support probability measures. 
If the geometric AEM generated by $Q$ satisfies the subset condition (i.e., is a good AEM), then there exists a minimal homeomorphism $T \in \Homeo(X)$ such that $Q = M_T(X)$.
\end{thm}

\begin{proof}
By Theorem~\ref{thm-GQ-dimension-group}, $(G_Q,u)$ is a simple dimension group
with state space $Q$.  By the HPS realization theorem \cite{HPS-92}, there is
a Cantor minimal Bratteli--Vershik system $(X_B,T_B)$ and an order-unit
isomorphism
\[
\theta:G_Q\longrightarrow K_0(X_B,T_B).
\]
Let
\[
Q_B=M_{T_B}(X_B).
\]
Since $G_Q$ is identified with a subgroup of $\Aff(Q)$, it has no nonzero
infinitesimals.  Hence neither does $K_0(X_B,T_B)$.  By Proposition~\ref{prop:k-theory-quotient}, the quotient
identification
\[
G_{Q_B}\cong K_0(X_B,T_B)/\Inf(K_0(X_B,T_B))
\]
therefore gives an order-unit isomorphism
\[
\Phi:G_Q\longrightarrow G_{Q_B}.
\]

Apply Lemma~\ref{lem:clopen_lifting} to $\Phi$.  We obtain a homeomorphism
\[
h:X\longrightarrow X_B
\]
such that
\[
[\mathbf 1_{h(A)}]=\Phi([\mathbf 1_A])
\qquad(A\in \CO(X)).
\]
We claim that
$h_*Q=Q_B$.

Indeed, if $\nu\in Q_B$, then the state $s_\nu$ on $G_{Q_B}$ pulls back via
$\Phi$ to a state on $G_Q$, hence by Theorem~\ref{thm-GQ-dimension-group}
there is a unique $\mu\in Q$ such that
\[
s_\nu\circ\Phi=s_\mu.
\]
For every clopen $A\subset X$,
\[
\nu(h(A))
=s_\nu([\mathbf 1_{h(A)}])
=s_\nu(\Phi([\mathbf 1_A]))
=\mu(A),
\]
so $h^{-1}_*\nu=\mu\in Q$.  Applying the same argument to $\Phi^{-1}$ gives
the reverse inclusion, and therefore $h_*Q=Q_B$.

Finally set
\[
T=h^{-1}\circ T_B\circ h.
\]
Then $T$ is minimal and
\[
M_T(X)=h^{-1}_*Q_B=Q.
\]
\end{proof}

As a consequence of this isomorphism, the Bratteli-Vershik system $(X_B, T_B)$ constructed from $G_Q$ in Theorem \ref{thm-realization} represents the canonical Orbit Equivalence class of $(X, T)$, as established by Giordano, Putnam, and Skau \cite{GPS-95}.

In particular, this generalizes Akin's theorem: if $Q = \{\mu\}$ is a singleton and $\mu$ satisfies the subset condition, then $\mu$ is the unique invariant measure for a minimal homeomorphism $T$.

\section{Fair Measures and Rigidity}\label{sec-fair}

In light of the realization theorem (Theorem \ref{thm-realization}), it is natural to ask whether the condition of being the invariant measure for a minimal homeomorphism can be intrinsically characterized for a single measure. We formalize this property as follows.

\begin{defn}
Let $X$ be a Cantor space. A probability measure $\mu \in M_{fc}(X)$ is called \emph{fair} if there exists a minimal homeomorphism $T \in \Homeo(X)$ such that $\mu \in M_T(X)$.
\end{defn}

By virtue of our previous results, the property of being fair can be tautologically restated in terms of good Choquet simplices.

\begin{prop}
A measure $\mu \in M_{fc}(X)$ is fair if and only if there exists a good Choquet simplex $Q \subset M_{fc}(X)$ such that $\mu \in Q$.
\end{prop}

\begin{proof}
If $\mu$ is fair, there is a minimal homeomorphism $T \in \Homeo(X)$ fixing $\mu$. By the general theory of minimal systems on Cantor spaces, the space of invariant measures $M_T(X)$ is a Choquet simplex $Q$ of full-support, atomless probability measures, and $\mu \in Q$, and by \cite{GW-95} $Q$ is good. 

Conversely, if such a good simplex $Q$ exists with $\mu \in Q$, Theorem \ref{thm-realization} provides a minimal homeomorphism $T \in \Homeo(X)$ such that $M_T(X) = Q$. Hence, $\mu$ is $T$-invariant, meaning $\mu$ is fair.
\end{proof}

Given this proposition, one might hope for a direct, purely measure-theoretic condition on $\mu$ itself that determines fairness, bypassing the need to identify an ambient simplex $Q$. However, such an intrinsic characterization is highly unlikely to exist. The symmetries required to support a minimal action are rare; for a typical measure, the group of measure-preserving homeomorphisms is in fact trivial.

\begin{thm}
For a generic measure $\mu \in M_{fc}(X)$ (in the sense of Baire category), the stabilizer group 
$$ \Homeo(X, \mu) = \{ h \in \Homeo(X) :  h_*\mu = \mu \} $$
is trivial. In particular, a generic measure is not fair.
\end{thm}

\begin{proof}
The space $M_{fc}(X)$ of atomless full-support probability measures on the Cantor space $X$, equipped with the weak$^*$ topology, is completely metrizable and therefore a Baire space. 

For any pair of distinct clopen sets $A, B \subset X$ (meaning their symmetric difference is non-empty), consider the set of measures that evaluate them equally:
$$ E_{A,B} = \{\nu \in M_{fc}(X) : \nu(A) = \nu(B)\}. $$
Because the evaluation maps $\nu \mapsto \nu(A)$ are continuous and affine, $E_{A,B}$ is a relatively closed affine slice of $M_{fc}(X)$. 

To see that this slice has empty interior, let $\nu \in E_{A,B}$. Since $A \neq B$, we may assume without loss of generality that there is a non-empty clopen subset $C \subset A \setminus B$. Because $\nu$ has full support, $\nu(C) > 0$. We can perturb $\nu$ to a new measure $\nu' \in M_{fc}(X)$ by shifting a small mass $\delta > 0$ from $X \setminus C$ into $C$. This perturbation yields $\nu'(A) = \nu(A) + \delta$ while $\nu'(B) \le \nu(B)$, pulling the measure strictly off the slice $E_{A,B}$. 

Therefore, every slice $E_{A,B}$ is nowhere dense. By the Baire Category Theorem, the countable union $\bigcup_{A \neq B} E_{A,B}$ over all pairs of distinct clopen sets is meager. Its complement is a dense $G_\delta$ set, meaning that a generic measure in $M_{fc}(X)$ strictly separates all distinct clopen sets.

If such a measure $\mu$ is preserved by a homeomorphism $h$, then
$\mu(h(A))=\mu(A)$ for every clopen $A$.  Since distinct clopen sets have
distinct $\mu$-measure, it follows that $h(A)=A$ for every clopen $A$.
Hence $h$ is the identity.  Thus the stabilizer of a generic measure is
trivial.
\end{proof}

See \cite{Akin-99} for explicite examples of rigid measures.

The rigidity of generic measures illustrates why an intrinsic condition for fairness is elusive. It underscores that fairness is not a local metric property of $\mu$, but rather depends on the measure being embeddable into a larger, highly structured dynamical simplex. 

Even if we isolate a fair measure, it may not possess the geometric ``goodness'' of its ambient system. One might mistakenly assume that if $\mu$ is fair, the singleton simplex $\{\mu\}$ must itself satisfy the subset condition. As we will explicitly construct in the next section (see Example \ref{ex-Q2}), a measure can be fair, yet fail the subset condition when viewed in isolation.

\section{Minimality of the Stabilizer Group and Good Simplices}
\label{sec-stabilizer}

In Theorem \ref{thm-realization}, we established that if
$Q\subset M_{fc}(X)$ is a good Choquet simplex, then there exists a
minimal homeomorphism $T\in\Homeo(X)$ such that
$Q=M_T(X)$.
It is natural to ask whether a converse may hold in terms of the full
stabilizer group
\[
\mathcal H_Q
=
\left\{
h\in\Homeo(X):
h_*\mu=\mu\text{ for every }\mu\in Q
\right\}.
\]
For a good simplex $Q$, Proposition \ref{prop-homog} and
Lemma \ref{lem-minimal} show directly that $\mathcal H_Q$ acts minimally
on $X$.  Does the converse hold?

Before addressing this question, it is useful to separate three properties
which are related but distinct: goodness of a measure, fairness, and
ergodicity for a given dynamical system.  Akin's Bernoulli example gives a
particularly instructive illustration.

\begin{exa}[Akin's Bernoulli example]\label{ex-Akin-Bernoulli}
In \cite[Theorem 2.18]{Akin-05}, Akin considers the Bernoulli measures
\[
\mu=\beta(1/3,2/3)
\qquad\text{and}\qquad
\nu=\beta(1/3,1/3,1/3)
\]
on the corresponding Bernoulli Cantor spaces.  He shows that
$S(\mu)=S(\nu)$,
and that this common clopen-values set is group-like, while $\nu$ is good
and $\mu$ is not good.  Moreover, Akin shows that the full group
\[
\mathcal H_\mu
=
\{h\in\Homeo(X):h_*\mu=\mu\}
\]
has a fixed point.

It follows immediately that $\mu$ is not fair: if a minimal homeomorphism
$T$ preserved $\mu$, then $T\in\mathcal H_\mu$, contradicting the existence
of a common fixed point for $\mathcal H_\mu$.  On the other hand, $\mu$ is
ergodic for the Bernoulli shift.  Thus ergodicity for a natural
measure-preserving transformation, and even the group-like property of
$S(\mu)$, do not imply fairness or goodness.
\end{exa}

As formulated, the stabilizer question also has a negative answer.  In fact,
even a singleton $Q=\{\mu\}$ may have a minimally acting stabilizer group while
failing the subset condition.  The following two elementary examples use
simple dimension groups.  For a systematic treatment of Akin's notion of
goodness in terms of traces on dimension groups, and for many related
examples and characterizations, see Bezuglyi and Handelman \cite{BH-14}.

The following two examples provide a
negative answer to Akin's question \cite[Theorem 2.15 and the following remark]{Akin-05}. 

\begin{exa}\label{ex-Q2}
Let
\[
G=\mathbb Q^2,
\qquad
G^+
=
\{0\}\cup
\{(x,y):x>0,\ y>0\},
\qquad
u=(1,1).
\]
This is a simple dimension group with order unit $u$.  By the HPS
realization theorem, there exists a Cantor minimal system $(X,T)$ such
that
\[
(K_0(X,T),K_0(X,T)^+,[\mathbf 1_X])
\cong
(G,G^+,u).
\]

The normalized state space of $G$ is parametrized by $t\in[0,1]$:
\[
\tau_t(x,y)=tx+(1-t)y.
\]
Choose an irrational $0<t<\frac12$.
Under the usual identification of states of $K_0(X,T)$ with
$T$-invariant probability measures, $\tau_t$ determines an invariant
measure on $X$.  Since $T$ is minimal, this measure has full support,
and since it is preserved by a minimal homeomorphism, it is fair.

Consider
\[
a=\left(\frac45,\frac15\right),
\qquad
b=\left(\frac15,\frac45\right).
\]
Both $a$ and $b$ belong to the order interval $[0,u]$.  By the
standard dimension-range property for Cantor minimal systems, there
exist clopen sets $A,B\subset X$ such that
\[
[\mathbf 1_A]=a,
\qquad
[\mathbf 1_B]=b
\]
in $K_0(X,T)$.

Now
\[
\tau_t(a)
=
\frac{1+3t}{5},
\qquad
\tau_t(b)
=
\frac{4-3t}{5}.
\]
Since $t<1/2$,
\[
\tau_t(A)<\tau_t(B).
\]

Suppose that the singleton
$Q=\{\tau_t\}$
were good.  
By the subset condition there would exist a clopen set
$A'\subset B$ such that
$\tau_t(A')=\tau_t(A)$.
Let
$a'=[\mathbf 1_{A'}]\in G$.
Since $A'\subset B$, $a'\leq b$.
On the other hand,
$\tau_t(a')=\tau_t(a)$.
Because $t$ is irrational, the homomorphism
$\tau_t:\mathbb Q^2\longrightarrow\mathbb R$
is injective.  Hence $a'=a$.

It follows that $a\leq b$, so
\[
b-a
=
\left(-\frac35,\frac35\right)
\]
would belong to $G^+$.  This is impossible.  Therefore the singleton
$\{\tau_t\}$ does not satisfy the subset condition.

Thus $\tau_t$ is a fair measure which is not good when regarded as a
singleton simplex.
\end{exa}

The measure in Example \ref{ex-Q2} is not ergodic for the minimal
homeomorphism $T$, since $\tau_t$ is not an extreme point of the state
space when $0<t<1$.  One may therefore ask whether an \emph{ergodic} fair
measure must be good.  

Bezuglyi and Handelman \cite{BH-14} showed that even an ergodic invariant measure of a minimal Cantor system need not be good, thereby answering Akin's question negatively. We give below a particularly elementary dimension-group example illustrating this phenomenon.


\begin{exa}[An ergodic fair measure which is not good]
\label{ex-ergodic-not-good}
Choose numbers
\[
0<\alpha<\frac12<\beta<1
\]
with $\alpha$ irrational.  Let
\[
G=\mathbb Q^2,
\qquad
u=(1,1),
\]
and, for $t\in[\alpha,\beta]$, put
\[
\tau_t(x,y)=tx+(1-t)y.
\]
Define
\[
G^+
=
\{0\}
\cup
\left\{
g\in\mathbb Q^2:
\tau_t(g)>0
\text{ for every }t\in[\alpha,\beta]
\right\}.
\]

We first note that $(G,G^+,u)$ is a simple dimension group.  Indeed,
consider the injective map
\[
\iota:G\longrightarrow\mathbb R^2,
\qquad
\iota(g)=\bigl(\tau_\alpha(g),\tau_\beta(g)\bigr).
\]
Since $\alpha\neq\beta$, this map is induced by an invertible real
linear transformation, and hence $\iota(G)$ is dense in
$\mathbb R^2$.  Moreover,
\[
G^+\setminus\{0\}
=
\iota^{-1}\bigl((0,\infty)^2\bigr).
\]
It follows immediately that the group is unperforated and that every
nonzero positive element is an order unit.

To verify interpolation, let
\[
x_1,x_2\leq y_1,y_2.
\]
If one of these inequalities is an equality, the common element is
an interpolant.  Otherwise all four inequalities are strict, and the
open rectangle
\[
\left(
\max_i\tau_\alpha(x_i),
\min_j\tau_\alpha(y_j)
\right)
\times
\left(
\max_i\tau_\beta(x_i),
\min_j\tau_\beta(y_j)
\right)
\]
is nonempty.  Since $\iota(G)$ is dense in $\mathbb R^2$, it contains
a point of this rectangle, which gives an interpolant in $G$.

The normalized state space of $(G,u)$ is precisely
$\{\tau_t:t\in[\alpha,\beta]\}$.
Indeed, after applying $\iota$, every normalized positive
homomorphism is a convex combination of the two coordinate
functionals $\tau_\alpha$ and $\tau_\beta$.

By the HPS realization theorem, there is therefore a Cantor minimal
system $(X,T)$ whose ordered $K_0$-group is $(G,G^+,u)$ and whose
simplex of invariant measures is affinely homeomorphic to the interval
$[\alpha,\beta]$.

The endpoint state $\tau_\alpha$ is an extreme point of this simplex.
Consequently, the corresponding $T$-invariant probability measure is
ergodic.  It is also fair, since it is preserved by the minimal
homeomorphism $T$.

We claim that the singleton
$Q=\{\tau_\alpha\}$ is nevertheless not good.
Consider
\[
a=\left(\frac35,\frac25\right),
\qquad
b=\left(\frac25,\frac35\right).
\]
Notice that $a+b=u$.
In particular, $a,b\in[0,u]$.  By the dimension-range property,
choose clopen sets $A,B\subset X$ such that
\[
[\mathbf 1_A]=a,
\qquad
[\mathbf 1_B]=b.
\]
We have
\[
\tau_\alpha(a)=\frac{2+\alpha}{5},
\qquad
\tau_\alpha(b)=\frac{3-\alpha}{5},
\]
and hence, since $\alpha<1/2$,
$\tau_\alpha(A)<\tau_\alpha(B)$.

Suppose that $\{\tau_\alpha\}$ satisfied the subset condition.  Then
there would exist a clopen set $A'\subset B$ such that
$\tau_\alpha(A')=\tau_\alpha(A)$.
Let
$a'=[\mathbf 1_{A'}]\in G$.
Because $\alpha$ is irrational, $\tau_\alpha$ is injective on
$\mathbb Q^2$.  Thus, $a'=a$.

Since $A'\subset B$, we would have $a\leq b$.
But
\[
b-a
=
\left(-\frac15,\frac15\right),
\]
and
\[
\tau_\beta(b-a)
=
\frac{1-2\beta}{5}<0,
\]
because $\beta>1/2$.  Hence
\[
b-a\notin G^+,
\]
contradicting $a\leq b$.

Therefore $\{\tau_\alpha\}$ is not a good simplex.  We have thus
obtained an ergodic fair measure which fails the subset condition when
viewed as a singleton.
\end{exa}

The three examples may be summarized as follows:
\[
\begin{array}{c|c|c|c}
\text{measure} & \text{ergodicity in the displayed system}
& \text{fair} & \text{good as a singleton} \\ \hline
\beta(1/3,2/3)
& \text{ergodic for the Bernoulli shift}
& \text{no} & \text{no} \\
\tau_t,\ 0<t<1/2
& \text{not ergodic for the minimal }T
& \text{yes} & \text{no} \\
\tau_\alpha
& \text{ergodic for the minimal }T
& \text{yes} & \text{no}
\end{array}
\]
Thus ergodicity, fairness, and goodness are genuinely different notions.
In particular, Akin's example shows that ergodicity does not imply fairness,
while Example \ref{ex-ergodic-not-good} shows that even ergodicity together
with fairness does not imply goodness.

The two $\mathbb Q^2$ examples also answer the original stabilizer question.

\begin{prop}\label{prop:minimal-stabilizer-not-good}
There exists a singleton
\[
Q=\{\mu\}\subset M_{fc}(X)
\]
such that $\mathcal H_Q$ acts minimally on $X$, while $Q$ does not satisfy
the subset condition.  The measure $\mu$ may moreover be chosen to be
ergodic for a Cantor minimal homeomorphism.
\end{prop}

\begin{proof}
Let $Q=\{\tau_\alpha\}$
be the singleton from Example \ref{ex-ergodic-not-good}.  The minimal
homeomorphism $T$ constructed there preserves $\tau_\alpha$, and hence
$T\in\mathcal H_Q$.
Since every $T$-orbit is dense, every $\mathcal H_Q$-orbit is dense as well.
Thus $\mathcal H_Q$ acts minimally on $X$.

On the other hand, Example \ref{ex-ergodic-not-good} shows that $Q$
does not satisfy the subset condition.
\end{proof}

The preceding proposition shows that minimality of $\mathcal H_Q$ alone is
far too weak to characterize good simplices.  It does not, however,
determine the full simplex of $\mathcal H_Q$-invariant measures.  This
suggests a stronger question.

For a subgroup $H\leq\Homeo(X)$, write
\[
M_H(X)
=
\left\{
\nu\in M(X):
h_*\nu=\nu\text{ for every }h\in H
\right\}.
\]

\begin{question}\label{qu-stabilizer-good}
Suppose that $Q\subset M_{fc}(X)$ is a Choquet simplex such that
$\mathcal H_Q$ acts minimally and
$M_{\mathcal H_Q}(X)=Q$.
Must $Q$ satisfy the subset condition?
\end{question}

There is a useful reformulation of this question in terms of group
actions.  Suppose that a group $\Gamma$ acts minimally on $X$ and put
$Q=M_\Gamma(X)$.
Then, every $\gamma\in\Gamma$ preserves every measure in $Q$, and
therefore
$\Gamma\subseteq\mathcal H_Q$.
It follows that
\[
M_{\mathcal H_Q}(X)\subseteq M_\Gamma(X)=Q.
\]
The reverse inclusion follows from the definition of $\mathcal H_Q$, and
hence
$M_{\mathcal H_Q}(X)=Q$.

Thus, a negative answer to Question \ref{qu-stabilizer-good} would follow
from any minimal action whose full invariant-measure simplex fails the
subset condition.

The recent examples of Boldrini and Prasad \cite{BP-26} are relevant to
this question.  They construct topologically free minimal Cantor actions
without dynamical comparison, including examples possessing invariant
probability measures.  However, failure of dynamical comparison does not
by itself imply failure of the subset condition.  Dynamical comparison
requires the relevant clopen pieces to be moved by elements of the acting
group, whereas the subset condition only asks for the existence of a
clopen subset having the prescribed evaluation profile.  Thus the examples
of \cite{BP-26} do not, by themselves, answer Question
\ref{qu-stabilizer-good}.

It would be interesting to determine whether the invariant-measure simplex
in one of those examples fails the subset condition, or, more generally,
whether there exists a minimal group action $\Gamma\curvearrowright X$ for
which $Q=M_\Gamma(X)$ is not good.

This question is particularly natural for amenable groups, since amenability
guarantees the existence of invariant probability measures.  We therefore
conclude by examining what the preceding AEM theory says for minimal
amenable Cantor actions.

\section{Amenable Cantor actions and the AEM framework}

Let $\alpha:G\curvearrowright X$ be a minimal continuous action of a countable group on a Cantor space, and
write $Q=M_G(X)$ for the set of $G$-invariant Borel probability measures.

We first separate three different levels of input from the preceding theory:
\begin{enumerate}
\item facts depending only on the pair $(X,Q)$;
\item facts depending on the additional hypothesis that $Q$ is good;
\item facts which use HPS, GPS, or the $K_0$-theory of a minimal
      $\Z$-action.
\end{enumerate}
Only the third level is intrinsically $\Z$-dynamical.

\subsection{The invariant simplex}

\begin{prop}\label{prop:Q-proper}
Assume that $G$ is countable and amenable and that
$\alpha:G\curvearrowright X$ is minimal.  Then $Q=M_G(X)$ is a nonempty compact
metrizable Choquet simplex, and every $\mu\in Q$ is atomless and has full
support.  In particular,
\[
Q\subset M_{fc}(X),
\]
and $Q$ defines a proper geometric AEM in the sense introduced above.
\end{prop}

\begin{proof}
Amenability gives the existence of a $G$-invariant probability measure on
the compact $G$-space $X$.  Compactness and convexity of $Q$ are immediate,
and the standard ergodic decomposition theorem for countable group actions
identifies $Q$ as a Choquet simplex.

If $\mu\in Q$, then $\operatorname{supp}(\mu)$ is a nonempty closed
$G$-invariant subset of $X$.  Minimality therefore gives
$\operatorname{supp}(\mu)=X$.

Finally, if $\mu$ had an atom $x$ of mass $a>0$, then every point of the
$G$-orbit of $x$ would also have mass $a$.  Minimality on a Cantor space
forces this orbit to be infinite, contradicting finiteness of $\mu$.
Hence every $\mu\in Q$ is atomless.
\end{proof}

\begin{rmk}
Freeness is not needed in Proposition~\ref{prop:Q-proper}.  Amenability is
used only to guarantee that $Q$ is nonempty.  If a nonamenable group acts
minimally and nevertheless admits invariant probability measures, the same
AEM discussion applies to its nonempty invariant simplex.
\end{rmk}

Thus every minimal amenable Cantor action canonically produces a proper AEM
\[
A\longmapsto \widehat A,\qquad
\widehat A(\mu)=\mu(A),\quad \mu\in Q.
\]
What is \emph{not} automatic is the subset condition.

\begin{defn}
We say that the invariant simplex $M_G(X)$ of $\alpha$ is \emph{good} if
$\widehat A<\widehat B$
for clopen $A,B\subset X$ implies that there exists a clopen
$A'\subset B$ such that
$\widehat{A'}=\widehat A$.
Equivalently, $Q=M_G(X)$ is a good AEM in the terminology introduced above.
\end{defn}

\subsection{The purely AEM-theoretic conclusions}

Assume now that $Q=M_G(X)$ is good.  Then every conclusion proved
from the Boolean algebra of clopen sets, the compact simplex $Q$, and the
subset condition remains valid \emph{without any change}.  In particular:

\begin{enumerate}
\item The relation
\[
A\sim_Q B
\quad\Longleftrightarrow\quad
\widehat A=\widehat B
\]
has the Boolean refinement and decomposition properties of
Lemma~\ref{lem-quasi}.

\item Boolean homogeneity and minimality of the full stabilizer remain
available: Proposition~\ref{prop-homog} gives, for $A\sim_QB$, an element
$h\in\mathcal H_Q$ with $h(A)=B$, and Lemma~\ref{lem-minimal} shows that
$\mathcal H_Q$ acts minimally on $X$.

\item Moreover, Proposition~\ref{prop-HQ-measures} gives
$M_{\mathcal H_Q}(X)=Q$.
Thus, enlarging the original action to the full $Q$-stabilizer introduces no
new invariant probability measures.

\item The Boolean Rokhlin lemma, Proposition~\ref{prop-uniform-R}, remains
valid.  It constructs many mutually $Q$-equivalent clopen sets with
uniformly small $Q$-remainder.  This should not be confused with the
construction of F{\o}lner towers for the original $G$-action.

\item By Theorem~\ref{thm-GQ-dimension-group}, the ordered group
\[
G_Q=C(X,\Z)/N_Q,
\qquad
N_Q=
\left\{
f\in C(X,\Z):
\int f\,d\mu=0\text{ for all }\mu\in Q
\right\},
\]
with the strict order introduced above, is a simple dimension group whose
normalized state space is canonically $Q$.

\item The clopen-scale theorem, Proposition~\ref{prop-clopen-scale}, gives
\[
[0,u]_{G_Q}
=
\{[\mathbf 1_A]:A\in\CO(X)\}.
\]

\item Finally, Corollary~\ref{cor-JQ-NQ} gives
\[
J_Q=
\left\langle
\mathbf 1_A-\mathbf 1_B:A\sim_QB
\right\rangle
=N_Q.
\]
Thus every integer-valued relation invisible to all invariant measures is
generated by elementary relations between clopen sets having identical
$Q$-profiles.
\end{enumerate}

None of these statements involves the orbit structure of the original
action $\alpha$.
There is nevertheless one elementary connection with the action:
$\alpha(G)\subseteq\mathcal H_Q$.
Indeed, every $\alpha(g)$ preserves each measure in $Q$.  This inclusion is
important, but it is generally very far from equality.

\subsection{Where HPS and GPS enter}

The realization theorem, Theorem~\ref{thm-realization}, shows that every good AEM is dynamically realizable: there
exists a minimal homeomorphism
$T:X\longrightarrow X$ such that $M_T(X)=Q$.
The proof uses the Herman--Putnam--Skau realization theorem
\cite{HPS-92}.

When $Q$ arose originally as $M_G(X)$ for an amenable group action, the
homeomorphism $T$ is an \emph{auxiliary} $\Z$-action constructed from the
AEM.  There is no reason for the $T$-orbits to agree with the $G$-orbits.

Likewise, the use above of the Giordano--Putnam--Skau theory
\cite{GPS-95}, the ordered group $K_0(X,T)$, and the infinitesimal subgroup
of $K_0(X,T)$ concern the auxiliary minimal $\Z$-system $(X,T)$ (or an
original system when the acting group actually is $\Z$).  These conclusions
can not be interpreted as orbit-equivalence or $K$-theoretic statements
about the original $G$-action.

In particular, the identification
\[
G_Q\cong K_0(X,T)/\operatorname{Inf}(K_0(X,T))
\]
is a statement about a minimal homeomorphism $T$ with invariant simplex
$Q$, not a canonical identification of $G_Q$ with the dynamical
coinvariants or type semigroup of an arbitrary amenable action.

\section{Goodness and \texorpdfstring{$\mathbb Z$}{Z}-realizability of invariant measures}

The realization theorem has a useful interpretation for general
amenable Cantor actions.

\begin{prop}\label{prop:Z-realizable}
Let $\alpha:G\curvearrowright X$ be a minimal action of a countable amenable
group and put $Q=M_G(X)$.  The following are equivalent:
\begin{enumerate}
\item $Q$ is a good AEM;
\item there exists a minimal homeomorphism $T:X\to X$ such that
\[
M_T(X)=M_G(X).
\]
\end{enumerate}
\end{prop}

\begin{proof}
If $Q$ is good, Theorem~\ref{thm-realization} gives such a minimal
homeomorphism $T$.
Conversely, suppose that $M_T(X)=Q$ for a minimal homeomorphism $T$.  The
Glasner--Weiss comparison theorem for Cantor minimal $\Z$-systems
\cite{GW-95} says that whenever
\[
\mu(A)<\mu(B)
\qquad(\mu\in M_T(X)),
\]
there exists a clopen $A'\subset B$ with
\[
\mu(A')=\mu(A)
\qquad(\mu\in M_T(X)).
\]
Since $M_T(X)=Q$, this is exactly goodness of the AEM generated by $Q$.
\end{proof}

Melleray explicitly points out that the following weaker form of the
orbit-equivalence problem is open in general: given a minimal Cantor action
of a countable amenable group $G$, must there exist a minimal $\Z$-action with
the same invariant Borel probability measures?  He also notes that
dynamical comparison gives a positive answer \cite{Melleray-25}.

Proposition~\ref{prop:Z-realizable} therefore identifies this question
exactly with the following AEM question:

\begin{question}\label{q:all-good}
Must $M_G(X)$ be a good AEM for every minimal Cantor action of a countable
amenable group?
\end{question}

Even a positive answer to Question~\ref{q:all-good} would say only that the
measure simplex of the action can be realized by a minimal homeomorphism.
It would not identify the orbit structure of the two actions.

\end{document}